\documentclass[11pt]{amsart}
\usepackage{amsmath}
\usepackage{amssymb}
\usepackage{amsfonts}
\usepackage{amsthm}  
\usepackage{latexsym}
\usepackage{graphicx}    
\usepackage{xypic} 
\usepackage{tikz} 
\usepackage{color} 
\usepackage{mathdots}
\usepackage{mathrsfs} 
\usepackage{dutchcal} 
\usepackage{verbatim} 
\usepackage{pdfpages} 
\usepackage{mathtools} 
\usepackage[pdftex,linktocpage=true,pdfborder=true]{hyperref}
\usepackage[nameinlink]{cleveref} 

\crefname{thm}{theorem}{theorems}
\crefname{lem}{lemma}{lemmas}
\crefname{cor}{corollary}{corollaries}
\crefname{prop}{proposition}{propositions}
\crefname{defn}{definition}{definitions}
\crefname{eg}{example}{examples}
\crefname{xca}{exercise}{exercises}
\crefname{conj}{conjecture}{conjectures}
\crefname{rmk}{remark}{remarks}
\crefname{qst}{question}{questions}
\crefname{obs}{observation}{observations}

\newtheorem{thm}{Theorem}[section]
\newtheorem*{thm*}{Theorem} 
\newtheorem{lem}[thm]{Lemma}  
\newtheorem{cor}[thm]{Corollary}
\newtheorem*{cor*}{Corollary}
\newtheorem{prop}[thm]{Proposition}
\theoremstyle{definition}
\newtheorem{defn}[thm]{Definition}

\theoremstyle{remark}
\newtheorem{rmk}[thm]{Remark}

\newtheorem*{note}{Note} 
\numberwithin{equation}{section}

\newcommand{\ip}[2]{\langle #1 , #2 \rangle}    

\newcommand\diag{\operatorname{diag}}   
\newcommand\C{\mathbb C}    
\newcommand\N{\mathbb N}

\newcommand\style{\mathfrak}
\newcommand\tran{{}^t} 
\newcommand\st{\operatorname{St}} 
\newcommand\GL{\mathbf{GL}} 
\newcommand\G{\mathbf{G}} 
\newcommand\Hbf{\mathbf{H}} 
\newcommand\X{\mathbf{X}} 
\newcommand\of{\mathcal O_F} 
\newcommand\Exp{\mathcal{Exp}} 
\newcommand\wt{\widetilde} 

\DeclareMathOperator{\spn}{span} 
\DeclareMathOperator{\Int}{Int}  
\DeclareMathOperator{\Ad}{Ad}  
\DeclareMathOperator{\ind}{Ind} 
\DeclareMathOperator{\irr}{Irr} 
\DeclareMathOperator{\Hom}{Hom} 
\DeclareMathOperator{\Gal}{Gal} 

\DeclarePairedDelimiter{\floor}{\lfloor}{\rfloor}

\begin{document}
\title[Local unitary periods and relative discrete series]{Local unitary periods and relative discrete series}
\author[J.~M.~Smith]{Jerrod Manford Smith}
\address{University of Maine, Department of Mathematics \& Statistics, Orono, Maine}
\email{jerrod.smith@maine.edu}
\thanks{The author was partially supported by the NSERC, Canada Graduate Scholarship and the Ontario Graduate Scholarship programs.}

\date{June 21, 2018}
\dedicatory{}
\begin{abstract}
Let $F$ be a $p$-adic field ($p\neq 2$), let $E$ be a quadratic Galois extension of $F$, and let $n \geq 2$.
We construct representations in the discrete spectrum of the $p$-adic symmetric space $H \backslash G$, where $G = \mathbf{GL}_{2n}(E)$ and $H = \mathbf{U}_{E/F}(F)$ is a quasi-split unitary group over $F$.
\end{abstract}
\maketitle
\setcounter{tocdepth}{1}
\tableofcontents
\section{Introduction}
Let $F$ be a $p$-adic field ($p\neq 2$) and let $E$ be a quadratic Galois extension of $F$. 
Let $G=\GL_n(E)$ and let $H$ be the group of $F$-points of a quasi-split unitary group $\mathbf{U}_{E/F}$.
In this paper, we are concerned with constructing irreducible representations of $G$ that occur in the discrete spectrum\footnote{To make sense of $L^2_{\operatorname{disc}}(H\backslash G)$, one also needs to take the quotient by the centre $Z_G$ of $G$ and consider square integrable functions on the quotient $Z_G H \backslash G$.  Moreover, we must consider representations that admit a (unitary) central character.} $L^2_{\operatorname{disc}}(H\backslash G)$ of the $p$-adic symmetric space $H\backslash G$. 
Such representations are referred to as relative discrete series (RDS) representations for $H\backslash G$. 
Of course, the issue of characterizing the discrete spectrum is of interest more generally. 
One would eventually strive for a recipe to construct representations in $L^2_{\operatorname{disc}}(H\backslash G)$ for any connected reductive group $\G$ and symmetric subgroup $\Hbf = \G^\theta$, where $\theta$ is an $F$-involution of $\G$.
For the symmetric spaces $\GL_n(F)\times \GL_n(F) \backslash \GL_{2n}(F)$ and $\GL_n(F) \backslash \GL_n(E)$, the author has carried out a construction of RDS, analogous to the one in \Cref{thm-unitary-rds}, in \cite[Theorem 6.3]{smith2018}.
Sakellaridis and Venkatesh have considered (and answered) much more general questions in the harmonic analysis on $p$-adic spherical varieties in \cite{sakellaridis--venkatesh2017}; however, they do not give an explicit description of the discrete spectrum.
Understanding the discrete series for $p$-adic symmetric spaces is a natural first step towards the general picture for spherical varieties.
Moreover, it is known by work of Kato and Takano \cite{kato--takano2010}, that $H$-distinguished\footnote{See \Cref{defn-dist}.} discrete series representations of $G$ are automatically RDS.  We are thus interested in constructing RDS that do not occur in the discrete spectrum of $G$.
 
In the case of number fields, distinction by unitary groups, that is, non-vanishing of the period integral attached to $\mathbf{U}_{E/F}$, has deep connections with quadratic base change \cite{Arthur--Clozel-book}.
The study of global period integrals has largely been pioneered by H.~Jacquet and his collaborators (see, for instance, \cite{jacquet--lai1985,jacquet2001,jacquet2005,jacquet2010,jacquet--lapid--rogawski1999}).
We refer the reader to the work of Feigon, Lapid and Offen \cite{feigon--lapid--offen2012} for a detailed discussion of both local and global aspects of the theory, as well as a very nice treatment of the history of the subject and the contributions of H.~Jacquet, \textit{et al}.  We recall some of the results of \cite{feigon--lapid--offen2012} in \Cref{sec-inducing-ds}; however, we will only state those results that we require to prove our main theorem. We encourage the reader to consult the work of Feigon, Lapid and Offen for more details; including a discussion of the failure of local multiplicity-one \cite[$\S13$]{feigon--lapid--offen2012}.

We now give a statement of the main result of the paper.
Let $n \geq 2$ be an integer.
Let $Q = P_{(n,n)}$ be the upper-triangular parabolic subgroup of $\GL_{2n}(E)$, with standard Levi factorization $L= M_{(n,n)}$, $U = N_{(n,n)}$. 
\begin{thm*}[\Cref{thm-unitary-rds}]
Let $\pi = \iota_Q^G \tau$ be a parabolically induced representation, where $\tau = \tau' \otimes {}^\sigma \tau'$, and $\tau'$ is a discrete series representation of $\GL_{n}(E)$ such that $\tau'$ is not Galois invariant, i.e.,  $\tau'\ncong {}^\sigma \tau'$.
The representation $\pi$ is a relative discrete series representation for ${\bf U}_{E/F}(F)\backslash\GL_{2n}(E)$ that does not occur in the discrete series of $\GL_{2n}(E)$.
\end{thm*}

\Cref{thm-unitary-rds} is the direct analogue of \cite[Theorem 6.3]{smith2018} and the overall method of proof is the same.
The idea of the proof is to reduce the verification of the Relative Casselman's Criterion (\Cref{rel-casselman-crit}) of Kato and Takano to the usual Casselman's Criterion for the inducing discrete series.
In contrast to the work in \cite{smith2018}, here we can construct representations only by inducing from the maximal parabolic subgroup $P_{(n,n)}$ due to the nature of the symmetric space $H \backslash G$ and the (lack of) existence of $\theta$-elliptic Levi subgroups (\textit{cf}.~\Cref{theta-elliptic-defn}, \Cref{lem-A_T-std-theta-ell-Levi}).  Similarly, the lack of $\theta$-elliptic Levi subgroups prevents us from considering $\GL_n(E)$, when $n$ is odd.

\begin{rmk}
In light of recent work of Rapha\"el Beuzart-Plessis, announced in his 2017 Cours Peccot, our construction of RDS via \Cref{thm-unitary-rds} exhausts all relative discrete series for the symmetric pair ${\bf U}_{E/F}(F)\backslash\GL_{2n}(E)$ that lie outside of $L^2_{\operatorname{disc}}(G)$.  The remaining relative discrete series are the ${\bf U}_{E/F}(F)$-distinguished discrete series representations of $\GL_{2n}(E)$.
We discuss the exhaustion of the discrete spectrum in \Cref{sec-exhaustion}.
\end{rmk}
We also prove the following corollary to \Cref{thm-unitary-rds}.
\begin{cor*}[\Cref{cor-infinitely-many}]
Let $n\geq 4$ be an integer that is not an odd prime. There exist infinitely many equivalence classes of RDS representations of the form constructed in \Cref{thm-unitary-rds} and such that  the discrete series $\tau$ is not supercuspidal.
\end{cor*}

We now give an outline of the contents of the paper.  In \Cref{sec-notation}, we set notation and recall basic facts regarding parabolic induction and distinguished representations.
We review the Relative Casselman's Criterion, due to Kato and Takano \cite{kato--takano2010}, in \Cref{sec-back-rds}.  We also discuss the invariant forms on Jacquet modules constructed by Kato and Takano \cite{kato--takano2008}, and Lagier \cite{lagier2008}.
In \Cref{sec-symm-pblc}, we recall the required structural results on $p$-adic symmetric spaces.  In particular, we discuss $(\theta,F)$-split tori, $\theta$-split parabolic subgroups, $\theta$-elliptic Levi subgroups, and the relative root system.
The main results of the paper (\Cref{thm-unitary-rds} and \Cref{cor-infinitely-many}) are stated and proved in \Cref{sec-rds}.  In \Cref{sec-technical} we prove the remaining technical results needed to establish the main theorem.

\section{Notation and preliminaries}\label{sec-notation}
Let $F$ be a nonarchimedean local field of characteristic zero and odd residual characteristic. 
Let $\of$ be the ring of integers of $F$. 
Let $E$ be a quadratic Galois extension of $F$. 
Let $\sigma \in \Gal(E/F)$ the nontrivial element of the Galois group of $E$ over $F$. 

For now, let $\G$ be an arbitrary connected reductive group defined over $F$ and let $G = \G(F)$ denote the group of $F$-points.
We will restrict to the case that $G = \GL_n(E)$ in \Cref{sec-rds} (onwards).
We let $Z_G$ denote the centre of $G$, and let $A_G$ denote the $F$-split component of $Z_G$.
Let $\theta$ be an $F$-involution of $\G$.
Define $\Hbf=\G^\theta$ to be the (closed) subgroup of $\theta$-fixed points of $\G$.
The quotient $H\backslash G$ is a $p$-adic symmetric space.

We will routinely abuse notation and identify an algebraic group defined over $F$ with its group of $F$-points. 
When the distinction is to be made, we will use boldface to denote the algebraic group and regular typeface to denote the group of $F$-points.
For any $F$-torus $\mathbf{A}$, we let $A^1$ denote the group of $\of$-points of $\mathbf A$.

Let $\GL_n$ denote the general linear group of $n$ by $n$ invertible matrices. We write  $\mathbf P_{(\underline{m})}$ for the block-upper triangular parabolic subgroup of $\GL_n$, corresponding to a partition $(\underline{m}) = (m_1,\ldots, m_k)$ of $n$.  The group $\mathbf P_{(\underline{m})}$ has block-diagonal Levi subgroup $\mathbf M_{(\underline{m})} \cong \prod_{i=1}^k \GL_{m_i}$ and unipotent radical $\mathbf N_{(\underline{m})}$.  
We use $\diag(a_1,a_2, \ldots, a_n)$ to denote an $n \times n$ diagonal matrix with entries $a_1,\ldots, a_n$.

For any $g,x \in G$, we write ${}^g x = gxg^{-1}$. For any subset $X$ of $G$, we write ${}^g X = \{ {}^g x : x \in X\}$.
Let $C_G(X)$ denote the centralizer of $X$ in $G$ and let $N_G(X)$ be the normalizer of $X$ in $G$.
Given a real number $r$ we let $\floor{r}$ denote the greatest integer that is less than or equal to $r$. 
We use $\widehat {(\cdot)}$ to denote that a symbol is omitted.  For instance, $\diag(\widehat{a_1}, a_2, \ldots, a_n)$ may be used to denote the diagonal matrix $\diag(a_2,\ldots, a_n)$.
\subsection{Induced representations of $p$-adic groups}
We now briefly review some necessary background of the representation theory of $G$ and discuss the representations that are relevant in the harmonic analysis on $H\backslash G$.
We will only consider representations on complex vector spaces. 
A representation $(\pi,V)$ of $G$ is smooth if for every $v\in V$ the stabilizer of $v$ in $G$ is an open subgroup.  
A smooth representation $(\pi,V)$ of $G$ is admissible if, for every compact open subgroup $K$ of $G$, the subspace $V^K$ of $K$-invariant vectors is finite dimensional. 
All of the representations that we consider are smooth and admissible. 
A quasi-character of $G$ is a one-dimensional representation.  
Let $(\pi,V)$ be a smooth representation of $G$. If $\omega$ is a quasi-character of $Z_G$, then $(\pi,V)$ is an $\omega$-representation if $\pi$ has central character $\omega$.  

Let $P$ be a parabolic subgroup of $G$ with Levi subgroup $M$ and unipotent radical $N$.  
 Given a smooth representation $(\rho, V_\rho)$ of $M$ we may inflate $\rho$ to a representation of $P$, also denoted $\rho$, by declaring that $N$ acts trivially.
 We define the representation $\iota_P^G \rho$ of $G$ to be the (normalized) parabolically induced representation $\ind_P^G (\delta_P^{1/2} \otimes \rho)$.
We normalize by the square root of the modular character  $\delta_P$ of $P$.  The character  $\delta_P$ is given by $\delta_P(p) = | \det \Ad_{\style n}(p) |$, for all $p\in P$, where $\Ad_{\style n}$ denotes the adjoint action of $P$ on the Lie algebra $\style n$ of $N$ \cite{Casselman-book}.
Let $(\pi,V)$ be a smooth representation of $G$.  Let $(\pi_N, V_N)$ denote the normalized Jacquet module of $\pi$ along $P$.
Precisely,  $V_N$ is the quotient of $V$ by the $P$-stable subspace $V(N) = \spn\{ \pi(n)v - v : n\in N , v\in V\}$, and the action of $P$ on $V_N$ is normalized by $\delta_P^{-1/2}$. 
The unipotent radical of $N$ acts trivially on $(\pi_N,V_N)$ and we will regard $(\pi_N, V_N)$ as a representation of the Levi factor $M \cong P/ N$ of $P$.

The Geometric Lemma (\Cref{geom-lem}) is a fundamental tool in the study of induced representations.
Let $P=MN$ and $Q=LU$ be two parabolic subgroups of $G$ with Levi factors $M$ and $L$, and unipotent radicals $N$ and $U$ respectively.
Following \cite{Roche2009}, let
\begin{align*}
S(M,L) = \{ y \in G : M \cap {}^y L\  \text{contains a maximal $F$-split torus of $G$}\}.
\end{align*}
There is a canonical bijection between the double-coset space $P \backslash G / Q$ and the set $M \backslash S(M,L) / L$.
Let $y \in S(M,L)$. The subgroup $M \cap {}^y Q$ is a parabolic subgroup of $M$ and  $P \cap {}^yL$ is a parabolic subgroup of ${}^yL$.
The unipotent radical of $M \cap{}^yQ$ is $M \cap {}^yU$ and the unipotent radical of $P \cap {}^yL$ is $N \cap {}^yL$; moreover, $M \cap {}^yL$ is a Levi subgroup of both $M \cap {}^y Q$ and $P \cap {}^yL$.
Given a representation $\rho$ of $L$, we obtain a representation ${}^y \rho = \rho \circ \Int y^{-1}$ of ${}^yL$.
\begin{lem}[The Geometric Lemma, {\cite[Lemma 2.12]{bernstein--zelevinsky1977}}]\label{geom-lem}
Let $\rho$ be a smooth representation of $L$.  There is a filtration of the space of the representation $(\iota_Q^G \rho)_N$ such that the associated graded object is isomorphic to the direct sum
\begin{align}\label{geom-lem-graded-obj}
\bigoplus_{y \in M \backslash S(M,L) / L} \iota_{M\cap {}^yQ}^M \left( ({}^y\rho)_{N\cap{}^y L} \right).
\end{align}
\end{lem}

\begin{rmk}\label{rmk-geom-lemma-notation}
We will write $\mathcal F_N^y(\rho)$ to denote the smooth representation $\iota_{M\cap {}^yQ}^M \left( ({}^y\rho)_{N\cap{}^y L} \right)$ of $M$.
\end{rmk}

Let $\Phi$ be the root system of $G$ (relative to a choice of maximal $F$-split torus $A_0$).  Fix a base $\Delta_0$ for $\Phi$ and let $W_0$ be the Weyl group of $G$ (with respect to $A_0$).  The choice of $\Delta_0$ determines a system $\Phi^+$ of positive roots.  Given a subset $\Theta$ of $\Delta_0$, we may associate a standard parabolic subgroup $P_\Theta = M_\Theta N_\Theta$, with Levi factor $M_\Theta$ and unipotent radical $N_\Theta$, in the usual way.
The following lemma gives a good choice of Weyl group representatives to use when applying \Cref{geom-lem} for standard parabolic subgroups.
\begin{lem}[{\cite[Proposition 1.3.1]{Casselman-book}}]\label{lem-nice-reps}
Let $\Theta$ and $\Omega$ be subsets of $\Delta_0$.
The set 
\begin{align*}
[W_\Theta \backslash W_0 /W_\Omega] = \{ w \in W_0 : w\Omega, w^{-1}\Theta \subset \Phi^+ \}
\end{align*}
provides a choice of Weyl group representatives for the double-coset space $P_\Theta \backslash G / P_\Omega$.
\end{lem}

We will always use the choice of ``nice" representatives  $[W_\Theta \backslash W_0 / W_\Omega] \subset S(M_\Theta,M_\Omega)$ for the double-coset space $P_\Theta \backslash G / P_\Omega \simeq M_\Theta \backslash S(M_\Theta, M_\Omega) / M_\Omega$.

\subsection{Distinguished (induced) representations}\label{sec-dist-reps}
Let $\pi$ be a smooth representation of $G$.  We also let $\pi$ denote its restriction to $H$.   Let $\chi$ be a quasi-character of $H$.  
\begin{defn}\label{defn-dist}
The representation $\pi$ is $(H,\chi)$-distinguished if the space $\Hom_H(\pi,\chi)$ is nonzero.  
\end{defn}
If $\pi$ is $(H,1)$-distinguished, where $1$ is the trivial character of $H$, then we will simply call $\pi$ $H$-distinguished. The elements of $\Hom_H(\pi,1)$ are $H$-invariant linear forms on the space of $\pi$.
\subsubsection{Relative matrix coefficients}
Let $(\pi,V)$ be a smooth $H$-distinguished representation of $G$.  
Let $\lambda \in \Hom_H(\pi,1)$ be a nonzero $H$-invariant linear form on $V$ and let $v$ be a nonzero vector in $V$.
In analogy with the usual matrix coefficients, define a complex-valued function $\varphi_{\lambda,v}$ on $G$ by
$\varphi_{\lambda,v}(g) = \ip{\lambda}{\pi(g)v}$.  
We refer to the functions $\varphi_{\lambda,v}$ as $\lambda$-relative matrix coefficients.
When $\lambda$ is understood, we will drop it from the terminology.
The representation $\pi$ is smooth; therefore, the relative matrix coefficients $\varphi_{\lambda,v}$ lie in $C^\infty(G)$, for every $v\in V$.
In addition, since $\lambda$ is $H$-invariant, the functions $\varphi_{\lambda,v}$ descend to well-defined functions on the quotient $H\backslash G$.

Let $\omega$ be a unitary character of $Z_G$ and further suppose that $\pi$ is an $\omega$-representation.
Since the central character $\omega$ is unitary, the function $Z_GH\cdot g \mapsto |\varphi_{\lambda,v}(g)|$ is well defined on $Z_GH\backslash G$.  The centre $Z_G$ of $G$ is unimodular since it is abelian.  
The fixed point subgroup $H$ is also reductive (\textit{cf.}~\cite[Theorem 1.8]{digne--michel1994}) and thus unimodular. 
It follows that there exists a $G$-invariant measure on the quotient $Z_GH \backslash G$ by \cite[Proposition 12.8]{Robert-book}.
\begin{defn}
The representation $(\pi,V)$ is said to be
\begin{enumerate}
\item $(H,\lambda)$-relatively square integrable if and only if all of the $\lambda$-relative matrix coefficients are square integrable modulo $Z_GH$.  
\item $H$-relatively square integrable if and only if $\pi$ is $(H,\lambda)$-relatively square integrable for every $\lambda \in \Hom_H(\pi,1)$.
\end{enumerate}
\end{defn}

When $H$ is understood, we drop it from the terminology and speak of relatively square integrable representations.
If $(\pi,V)$ is $H$-distinguished and $(H,\lambda)$-relatively square integrable, then the morphism that sends $v\in V$ to the relative matrix coefficient $\varphi_{\lambda,v}$ is an intertwining operator from $\pi$ to the right regular representation of $G$ on $L^2(Z_GH \backslash G, \omega)$, where $L^2(Z_GH \backslash G, \omega)$ is the space of  functions on $H \backslash G$, square integrable modulo $Z_G$, that are $Z_G$-eigenfunctions with eigencharacter $\omega$.

\begin{defn}
If $(\pi,V)$ is an irreducible subrepresentation of $L^2(Z_GH\backslash G)$, then we say that $(\pi,V)$ occurs in the discrete spectrum of $H\backslash G$.   
In this case, we say that $(\pi,V)$ is a relative discrete series (RDS) representation.
\end{defn}

The main goal of the present paper is to construct RDS representations for $\mathbf{U}_{E/F}(F) \backslash \GL_{2n}(E)$, when $\mathbf{U}_{E/F}$ is quasi-split unitary group over $F$.
\subsubsection{Invariant forms on induced representations}
\Cref{lem-hom-injects} is well known and follows from an explicit version of Frobenius Reciprocity due to Bernstein and Zelevinsky \cite[Proposition 2.29]{bernstein--zelevinsky1976}.  Let $Q = LU$ be a $\theta$-stable parabolic subgroup with $\theta$-stable Levi factor $L$ and unipotent radical $U$.  Note that the identity component of $Q^\theta = L^\theta U^\theta$ is a parabolic subgroup of the identity component $H^\circ$ of $H$, with the expected Levi decomposition (\textit{cf.}~\cite{helminck--wang1993}, \cite[Lemma 3.1]{gurevich--offen2016}).  Let $\mu$ be a positive quasi-invariant measure on the  quotient $Q^\theta \backslash H$ \cite[Theorem 1.21]{bernstein--zelevinsky1976}.

\begin{lem}\label{lem-hom-injects}
Let $\rho$ be a smooth representation of $L$ and let $\pi = \iota_Q^G \rho$. 
The map $\lambda \mapsto \lambda^G$ is an injection of $\Hom_{L^\theta}( \delta_Q^{1/2}\rho, \delta_{Q^\theta})$ into $\Hom_H(\pi,1)$, where for any function $\phi$ in the space of $\pi$, $\lambda^G$ is given explicitly by
\begin{align*}
\ip{\lambda^G}{\phi} &= \int_{Q^\theta \backslash H} \ip{\lambda}{\phi(h)} \ d\mu(h).
\end{align*}
\end{lem}

\begin{cor}\label{cor-hom-injects}
If $\delta_Q^{1/2}$ restricted to $L^\theta$ is equal to $\delta_{Q^\theta}$, then the map $\lambda \mapsto \lambda^G$ is an injection of $\Hom_{L^\theta}(\rho, 1)$ into $\Hom_H(\pi,1)$. In particular, if $\rho$ is $L^\theta$-distinguished, then $\pi$ is $H$-distinguished.
\end{cor}

\begin{proof}
Observe that $\Hom_{L^\theta}(\delta_Q^{1/2}\rho, \delta_{Q^\theta}) = \Hom_{L^\theta}(\rho, \delta_Q^{-1/2}\vert_{L^\theta}\delta_{Q^\theta})$.
\end{proof}

In fact, the $H$-invariant linear form on $\pi = \iota_Q^G\rho$ arises from the closed orbit in $Q\backslash G / H$ via the Mackey theory.
\section{Background on RDS: the Relative Casselman's Criterion}\label{sec-back-rds}
\subsection{Exponents (of induced representations)}
Let $(\pi, V)$ be a finitely generated admissible representation of $G$.  
Let $\chi$ be a quasi-character of the $F$-split component $A_G$ of the centre of $G$. 
For $n\in \N$, $n\geq 1$, define the subspace
\begin{align*}
V_{\chi, n} = \{ v \in V : (\pi(z) - \chi(z))^n v = 0, \ \text{for all} \ z\in A_G \},
\end{align*}
and set 
\begin{align*}
V_{\chi} = \bigcup_{n=1}^\infty V_{\chi, n}.
\end{align*}
Each $V_{\chi, n}$ is a $G$-stable subspace of $V$ and $V_{\chi}$ is the generalized $\chi$-eigenspace in $V$ for the $A_G$-action on $V$.  
By \cite[Proposition 2.1.9]{Casselman-book},
\begin{enumerate}
\item $V$ is a direct sum $\displaystyle V = \bigoplus_{\chi} V_{\chi}$, where $\chi$ ranges over quasi-characters of $A_G$, and \label{item-direct}
\item since $V$ is finitely generated, there are only finitely many $\chi$ such that $V_{\chi} \neq 0$. Moreover, there exists $n\in \N$ such that $V_{\chi} = V_{\chi,n}$, for each $\chi$. \label{item-finite}
\end{enumerate}
Let $\Exp_{A_G}(\pi)$ be the (finite) set of quasi-characters of $A_G$ such that $V_{\chi} \neq 0$.  The quasi-characters that appear in $\Exp_{A_G}(\pi)$ are called the exponents of $\pi$.
The second item above implies that $V$ has a finite filtration such that the quotients are $\chi$-representations, for $\chi \in \Exp_{A_G}(\pi)$.

\begin{lem}\label{exp-irred-subq}
The characters $\chi$ of $A_G$ that appear in $\Exp_{A_G}(\pi)$ are the central quasi-characters of the irreducible subquotients of $\pi$.
\end{lem}

Let $(\pi,V)$ be a finitely generated admissible  representation of $G$. Let $P=MN$ be a parabolic subgroup of $G$ with Levi factor $M$ and unipotent radical $N$.  
It is a theorem of Jacquet that $(\pi_N, V_N)$ is also finitely generated and admissible (\textit{cf.}~\cite[Theorem 3.3.1]{Casselman-book}).  Applying \eqref{item-direct} and \eqref{item-finite} to $(\pi_N, V_N)$, we obtain a direct sum decomposition
\begin{align*}
V_N = & \bigoplus_{\chi \in \Exp_{A_M}(\pi_N)} (V_N)_{\chi}
\end{align*}
where the set $\Exp_{A_M}(\pi_N)$ of quasi-characters of $A_M$, such that $(V_N)_{\chi} \neq 0$, is finite.
The quasi-characters of $A_M$ appearing in $\Exp_{A_M}(\pi_N)$ are called the exponents of $\pi$ along $P$.

We are ultimately interested in the exponents of parabolically induced representations.
For a proof of the following lemma, see \cite[Lemma 4.15]{smith2018}

\begin{lem}\label{red-to-ind-exp}
Let $P = MN$ be a parabolic subgroup of $G$, let $(\rho, V_\rho)$ be a  finitely generated admissible representation of $M$ and let $\pi = \iota_P^G\rho$.
The quasi-characters $\chi \in \Exp_{A_G}(\pi)$ are the restriction to $A_G$ of characters $\eta$ of $A_M$ appearing in $\Exp_{A_M}(\rho)$.
\end{lem}

\subsection{Invariant linear forms on Jacquet modules}\label{sec-r-P-lambda}
Let $(\pi,V)$ be an admissible $H$-distinguished representation of $G$. Let $\lambda$ be a nonzero $H$-invariant linear form on $V$. 
Let $P$ be a $\theta$-split parabolic subgroup of $G$.  Recall that a parabolic subgroup $P$ of $G$ is $\theta$-split if $\theta(P)$ is opposite to $P$ (\textit{cf.}~\Cref{sec-pblc}).  Let $N$ be the  unipotent radical of $P$, and let $M= P \cap \theta(P)$ be a $\theta$-stable Levi factor of $P$.  
Independently, Kato--Takano  and Lagier define an $M^\theta$-invariant linear form $\lambda_N$ on the Jacquet module $(\pi_N, V_N)$.
The construction of $\lambda_N$ relies on  Casselman's Canonical Lifting \cite[Proposition 4.1.4]{Casselman-book}.
We next record \cite[Proposition 5.6]{kato--takano2008}, and 
we refer the interested reader to  \cite{kato--takano2008,lagier2008} for the details of the construction of $\lambda_N$. 
\begin{prop}[Kato--Takano, Lagier]\label{rPlambda-prop}
Let $(\pi,V)$ be an admissible $H$-distinguished representation of $G$.
Let $\lambda \in \Hom_H(\pi,1)$ be nonzero and let $P$ be a $\theta$-split parabolic subgroup of $G$ with unipotent radical $N$ and $\theta$-stable Levi component $M= P \cap \theta(P)$.
\begin{enumerate}
\item The linear functional $\lambda_N: V_N \rightarrow \C$ is $M^\theta$-invariant.
\item The mapping $\Hom_H(\pi,1) \rightarrow \Hom_{M^\theta}(\pi_N, 1)$, sending $\lambda$ to $\lambda_N$, is linear.
\end{enumerate}
\end{prop}

\subsection{The Relative Casselman's Criterion}
Let $(\pi,V)$ be a finitely generated admissible $H$-distinguished representation of $G$. Fix a nonzero $H$-invariant form $\lambda$ on $V$.
 For any closed subgroup $Z$ of the centre of $G$, Kato and Takano \cite{kato--takano2010} define
 \begin{align}\label{relative-exponent}
 \Exp_Z(\pi,\lambda) & = \{ \chi \in \Exp_Z(\pi) : \lambda \vert_{V_{\chi}} \neq 0 \},
 \end{align}
 and refer to the set $\Exp_Z(\pi,\lambda)$ as exponents of $\pi$ relative to $\lambda$.
 
The following appears as \cite[Theorem 4.7]{kato--takano2010}, see \Cref{sec-tori-involution} for the definition of the set $S_M^- \setminus S_GS_M^1$.
\begin{thm}[The Relative Casselman's Criterion, Kato--Takano]\label{rel-casselman-crit}
Let $\omega$ be a unitary character of $Z_G$. Let $(\pi,V)$ be a finitely generated admissible $H$-distinguished $\omega$-representation of $G$.
Fix a nonzero $H$-invariant linear form $\lambda$ on $V$. The representation $(\pi,V)$ is $(H,\lambda)$-relatively square integrable  
if and only if the condition 
 \begin{align}\label{rel-casselman} 
|\chi(s)| &< 1 & \text{for all} \ \chi \in \Exp_{S_M}(\pi_N, \lambda_N) \ \text{and all} \ s\in S_M^- \setminus S_GS_M^1
 \end{align}
  is satisfied for every proper $\theta$-split parabolic subgroup $P=MN$ of $G$.
\end{thm}

It is an immediate corollary of \Cref{rel-casselman-crit} that: if $(\pi,V)$ is an $H$-distinguished discrete series representation of $G$, then $\pi$ is $H$-relatively square integrable.
For a proof of the following, see \cite[Proposition 4.22]{smith2018}.

\begin{prop}\label{non-dist-gen-eig-sp}
Let $(\pi, V)$ be a finitely generated admissible representation of $G$.  Let $\chi \in \Exp_{Z_G}(\pi)$ and assume that none of the irreducible subquotients of $(\pi,V)$ with central character $\chi$ are $H$-distinguished.  Then for any $\lambda \in \Hom_H(\pi,1)$, the restriction of $\lambda$ to $V_{\chi}$ is equal to zero, i.e., $\lambda \vert_{V_{\chi}} \equiv 0$.
\end{prop}

\section{$p$-adic symmetric spaces and parabolic subgroups}\label{sec-symm-pblc}
In this section, we discuss the tori, root systems, and parabolic subgroups relevant to our study of $H\backslash G$.
We briefly review some general notions before turning our attention to the case of $\mathbf{U}_{E/F}(F) \backslash \GL_{2n}(E)$ in \Cref{sec-unitary-structure}.
\subsection{$(\theta,F)$-split tori and the relative root system}\label{sec-tori-involution}
We say that an element $g \in G$ is $\theta$-split if $\theta(g) = g^{-1}$.
Recall that an $F$-torus $S$ contained in $G$ is $(\theta,F)$-split if $S$ is $F$-split and every element of $S$ is $\theta$-split.
Let $S_0$ be a maximal $(\theta,F)$-split torus of $G$.  Fix a $\theta$-stable maximal $F$-split torus $A_0$ of $G$ that contains $S_0$ \cite[Lemma 4.5(iii)]{helminck--wang1993}.
Let $\Phi_0 = \Phi(G,A_0)$ be the root system of $G$ with respect to $A_0$, and let $W_0$ be the associated Weyl group.

Let  $M$ be any Levi subgroup of $G$.  Let $A_M$ be the $F$-split component of the centre of $M$.
The $(\theta,F)$-split component of $M$ is the largest $(\theta,F)$-split torus $S_M$ contained in $Z_M$. 
The torus $S_M$ is the connected component of the subgroup of $\theta$-split elements in $A_M$. Explicitly,\footnote{Here $(\cdot)^\circ$ indicates the Zariski-connected component of the identity.}
\begin{align*}
S_M = \left( \{ x \in A_M : \theta(x) = x^{-1}\} \right)^\circ.
\end{align*}  

There is an action of $\theta$ on the $F$-rational characters $X^*(A_0)$ of $A_0$.  Indeed, since $A_0$ is $\theta$-stable, for $\chi \in X^*(A_0)$, the character
\begin{align*}
(\theta \chi)(a) = \chi(\theta(a))
\end{align*}
is well defined for all $a \in A_0$.
In addition, $\Phi_0 \subset X^*(A_0)$ is stable under the action of $\theta$.  Let $\Phi_0^\theta$ be the set of $\theta$-fixed roots.
Recall that a choice $\Delta_0$ of base for $\Phi_0$ determines a system $\Phi_0^+$ of positive roots.  \begin{defn}\label{defn-theta-base}
A base $\Delta_0$ of $\Phi_0$ is called a $\theta$-base if for every $\alpha \in \Phi_0^+$, such that $\alpha \neq \theta (\alpha)$, we have that $\theta(\alpha) \in \Phi_0^-$.   
\end{defn}
Let $\Delta_0$ be a $\theta$-base of $\Phi_0$ (existence of a $\theta$-base is proved in \cite{helminck1988}).
Let $p: X^*(A_0) \rightarrow X^*(S_0)$ be the morphism defined by restricting the $F$-rational characters of $A_0$ to the subtorus $S_0$.  The map $p$ is surjective and the kernel of $p$ is the submodule $X^*(A_0)^\theta$ consisting of $\theta$-fixed $F$-rational characters.
The restricted root system of $H \backslash G$ (relative to our choice of $(A_0,S_0,\Delta_0)$) is defined to be
\begin{align*}
\overline \Phi_0 = p(\Phi_0)\setminus \{0\} = p(\Phi_0 \setminus \Phi_0^\theta).
\end{align*}
The set $\overline \Phi_0$ coincides with the set $\Phi(G,S_0)$ of roots in $G$ with respect to $S_0$.  The set $\overline \Phi_0$ is a root system by \cite[Proposition 5.9]{helminck--wang1993}; however, $\overline \Phi_0$ is not necessarily reduced.  
The set
\begin{align*}
\overline \Delta_0 = p(\Delta_0) \setminus \{0\} = p(\Delta_0 \setminus \Delta_0^\theta)
\end{align*}
forms a base for $\overline \Phi_0$. 
The linear independence of $\overline \Delta_0$ follows from the fact that $\Delta_0$ is a $\theta$-base and that $\ker p = X^*(A_0)^\theta$ consists of $\theta$-fixed characters.

Given a subset $\overline \Theta \subset \overline \Delta_0$, define the subset 
\begin{align*}
[\overline\Theta ] = p^{-1}(\overline\Theta ) \cup \Delta_0^\theta
\end{align*}
of $\Delta_0$. 
Subsets of $\Delta_0$ of the form $[\overline\Theta ]$, where $\overline \Theta \subset \overline \Delta_0$, are called $\theta$-split.  
The maximal $\theta$-split subsets of $\Delta_0$ are of the form
$[\overline\Delta_0 \setminus\{\bar\alpha\}]$,  
where $\bar\alpha \in \overline\Delta_0$. 
\subsection{$\theta$-split parabolic subgroups and $\theta$-elliptic Levi factors}\label{sec-pblc}
As above, let $\Delta_0$ be a $\theta$-base of $\Phi_0$.
To any subset $\Theta$ of $\Delta_0$, we may associate a $\Delta_0$-standard parabolic subgroup $P_\Theta$ of $G$ with unipotent radical $N_\Theta$ and standard Levi factor $M_\Theta = C_G(A_\Theta)$, where $A_\Theta$ is the $F$-split torus
\begin{align*}
A_\Theta  = \left( \bigcap_{\alpha \in \Theta } \ker \alpha \right)^\circ.
\end{align*}
Let $\Phi_\Theta $ be the subsystem of $\Phi_0$ generated by the simple roots $\Theta$. 
Let $\Phi_\Theta ^+$ be the system of $\Theta$-positive roots.
The unipotent radical $N_\Theta $ of $P_\Theta $ is generated by the root groups $N_\alpha$, where $\alpha \in \Phi_0^+\setminus \Phi_\Theta ^+$.
The torus $A_\Theta $ is the $F$-split component of the centre of $M_\Theta $ and $\Phi_\Theta $ is the root system of $A_0$ in $M_\Theta$.

\begin{defn}
A parabolic subgroup $P$ of $G$ is $\theta$-split if $\theta(P)$ is opposite to $P$.
\end{defn}
If $P$ is a $\theta$-split parabolic subgroup, then $M = P \cap \theta(P)$ is a $\theta$-stable Levi subgroup of both $P$ and the opposite parabolic $P^{\operatorname{op}} = \theta(P)$.
If $\Theta \subset \Delta_0$ is $\theta$-split, then the $\Delta_0$-standard parabolic subgroup $P_\Theta  = M_\Theta N_\Theta$ is $\theta$-split. 
Any $\Delta_0$-standard $\theta$-split parabolic subgroup arises this way \cite[Lemma 2.5(1)]{kato--takano2008}.
Following \cite[$\S1.5$]{kato--takano2010}, the $(\theta,F)$-split component of $M_\Theta $ is equal to 
\begin{align*}
S_\Theta  = \left( \{ s \in A_\Theta  : \theta(s) = s^{-1} \} \right)^\circ =  \left( \bigcap_{ \bar\alpha \in p(\Theta)} \ker (\bar\alpha: S_0 \rightarrow F^\times )  \right)^\circ.
\end{align*}
For any $0 < \epsilon \leq 1$, define 
\begin{align}\label{eq-split-dominant-part}
S_\Theta ^-(\epsilon) = \{ s \in S_\Theta  : |\alpha(s)|_F \leq \epsilon, \ \text{for all} \ \alpha \in \Delta_0 \setminus \Theta \}.
\end{align}
Let $S_\Theta ^-$ denote $S_\Theta ^-(1)$. The set $S_\Theta ^-$ is referred to as the dominant part of $S_\Theta$.

By \cite[Theorem 2.9]{helminck--helminck1998}, the subset $\Delta_0^\theta$ of $\theta$-fixed roots in $\Delta_0$ determines the ($\Delta_0$-standard) minimal $\theta$-split parabolic subgroup $P_0 = P_{\Delta_0^\theta}$.  By \cite[Proposition 4.7(\rm{iv})]{helminck--wang1993}, the minimal $\theta$-split parabolic subgroup $P_0$ has standard $\theta$-stable Levi $M_0 = C_G(S_0)$.   
Let $N_0$ be the unipotent radical of $P_0$.  We have that $P_0 = M_0N_0$.

\begin{lem}[{\cite[Lemma 2.5]{kato--takano2008}}]\label{KT08-lem-2.5}
Let $S_0 \subset A_0$, $\Delta_0$, and $P_0 = M_0N_0$ be as above.
\begin{enumerate}
\item Any $\theta$-split parabolic subgroup $P$ of $G$ is conjugate to a $\Delta_0$-standard $\theta$-split parabolic subgroup by an element $g \in (\Hbf \mathbf M_0)(F)$.
\item If the group of $F$-points of the product $(\Hbf \mathbf M_0)(F)$ is equal to $HM_0$, then any $\theta$-split parabolic subgroup of $G$ is $H$-conjugate to a $\Delta_0$-standard $\theta$-split parabolic subgroup.
\end{enumerate}
\end{lem} 

Let $P = MN$ be a $\theta$-split parabolic subgroup. Pick $g\in (\Hbf \mathbf M_0)(F)$ such that $P = gP_\Theta  g^{-1}$ for some $\theta$-split subset $\Theta \subset \Delta_0$.  
Since $g\in (\Hbf \mathbf M_0)(F)$ we have that $g^{-1}\theta(g) \in \mathbf M_0(F)$, and
we have $S_M = g S_\Theta g^{-1}$.
For a given $\epsilon >0$, one may extend the definition of $S_\Theta ^-$ in \eqref{eq-split-dominant-part}  to the torus $S_M$. 
Set
$S_M^-(\epsilon) = g S_\Theta ^-(\epsilon) g^{-1}$ 
and define $S_M^- = S_M^-(1)$. Recall that we write $S_M^1$ to denote the $\of$-points $S_M(\of)$.

The next definition is made in analogy with the notion of an elliptic Levi subgroup.  The following terminology is due to Murnaghan \cite{murnaghan2017a}.  

\begin{defn}\label{theta-elliptic-defn}
A $\theta$-stable Levi subgroup $L$ of $G$ is $\theta$-elliptic if and only if $L$ is not contained in any proper $\theta$-split parabolic subgroup of $G$.
\end{defn}

The next lemma follows immediately from \Cref{theta-elliptic-defn}.

\begin{lem}\label{contain-theta-elliptic}
If a $\theta$-stable Levi subgroup $L$ of $G$ contains a $\theta$-elliptic Levi subgroup, then $L$ is $\theta$-elliptic.
\end{lem}

The next proposition appears in \cite[Proof of Proposition 8.4]{murnaghan2017a}.

\begin{prop}\label{theta-elliptic-theta-stable}
Let $Q$ be a parabolic subgroup of $G$. If $Q$ admits a $\theta$-elliptic Levi factor $L$, then $Q$ is $\theta$-stable.
\end{prop}

\begin{proof}
By definition, $L$ is $\theta$-stable.
One can show that for any root $\alpha$ of $A_L$ in $G$ we have $\theta (\alpha) = \alpha$.  
It follows that the unipotent radical of $Q$ is also $\theta$-stable.
\end{proof}
\subsection{Structure of $\mathbf{U}_{E/F}(F) \backslash \GL_{2n}(E)$}\label{sec-unitary-structure}
Let $\G = R_{E/F} \GL_n$ be the restriction of scalars from $E$ to $F$ of $\GL_n$. 
We will restrict to the case that $n$ is even in \Cref{sec-rds} (onward).
We identify the group $G = \G(F)$ with the set $\GL_n(E)$, of $E$-points of $\GL_n$.  The non-trivial element $\sigma$ of the Galois group $\Gal(E/F)$ gives rise to an $F$-involution of $\G$ given by entry-wise Galois conjugation on $\GL_n(E)$.  We denote the Galois involution of $\G$ by $\sigma$. Explicitly,
\begin{align*}
\sigma(g) 
& = (\sigma({g_{ij}})), & \text{where} \ g = (g_{ij}) \in G.
\end{align*}
Following \cite{feigon--lapid--offen2012}, let ${\X}$ denote the $F$-variety of Hermitian matrices in $\G$, 
\begin{align}
\label{Hermitian} {\X} = \{ x \in \G : \tran  \sigma(x) = x \}.
\end{align}
Here $\tran g$ denotes the transpose of $g\in \G$.
There is a right action of $\G$ on ${\X}$ given by $x \cdot g = \tran \sigma({g}) x g$, where $x\in \X$ and $g\in \G$.
Write $X = \X(F)$ for the $F$ points of $\X$.  There is a finite set $X/ G$ of $G$-orbits in $X$ indexed by $F^\times / N_{E/F}(E^\times)$ \cite{feigon--lapid--offen2012}.  By Local Class Field Theory,  $F^\times / N_{E/F}(E^\times)$ is isomorphic to $\Gal(E/F)$, and thus consists of two elements.

Given $x\in X$, define an $F$-involution $\theta_x$ of $G$ by
\begin{align}\label{x-involution}
\theta_x(g) = x^{-1} \tran  \sigma({g})^{-1} x,
\end{align}
for all $g\in G$.
Let $\Hbf^x = \G^{\theta_x}$ be the subgroup of $\theta_x$-fixed elements.  
The group of $F$-points $H^x = \Hbf^x(F)$ is a unitary group associated to $E/F$ and $x$.  

\begin{rmk}
In the literature, ${\bf U}_{E/F, x}$ is often used to denote the unitary group $\Hbf^x$ associated to $E/F$ and $x$.
We will use the ${\bf U}_{E/F, x}$ notation for unitary groups that appear as subgroups of Levi factors of $G$.
\end{rmk}

\begin{defn}\label{defn-involution-action}
An involution $\theta_1$ of $G$ is $G$-equivalent to another involution $\theta_2$ if there exists $g\in G$ such that $\theta_1 = \Int g^{-1} \circ \theta_2 \circ \Int g$, 
where $\Int g$ denotes the inner $F$-automorphism of $\G$ given by $\Int g(x) = g x g^{-1}$, for all $x\in \G$. 
We write $g\cdot \theta$ to denote the involution $\Int g^{-1} \circ \theta \circ \Int g$. 
\end{defn}

Two involutions $\theta_{x_1}$ and $\theta_{x_2}$ are $G$-equivalent if and only if $x_1$ and $x_2$ lie in the same $G$-orbit in $X/G$.  Indeed, if 
there exists $g \in G$ such that $y = x \cdot g = \tran  \sigma({g}) x g$, then one can check that $\theta_y$ is equal to the involution $g\cdot \theta_x = \Int g^{-1} \circ \theta_x \circ \Int g$.  Note that the $G$-action $\theta \mapsto g\cdot \theta$ on involutions is also a right-action.  Since $X/G$ has order two, there are two $G$-equivalence classes of involutions of the form $\theta_x$. It is well known that when $n$ is odd, $\Hbf^x$ is always quasi-split over $F$.  When $n$ is even there are two isomorphism classes of unitary group associated to $E/F$, one of which is quasi-split. 

We fix $\theta = \theta_{w_\ell}$, where $w_{\ell}$ is the permutation matrix in $G$ with unit anti-diagonal, 
and write $\Hbf = \G^\theta$.
The group $\Hbf = \mathbf{U}_{E/F,w_\ell}$ is quasi-split over $F$.    
Write $H = \Hbf(F)$ for the group of $F$-points of $\Hbf$.

Let $J_r$ be the $r\times r$ permutation matrix with unit anti-diagonal
\begin{align*}
J_r = \left( \begin{matrix} & & 1 \\ & \iddots & \\ 1 & & \end{matrix} \right).
\end{align*}
 and note that $w_\ell = J_n$. 
For any positive integer $r$, there exists $\gamma_r \in \GL_r(E)$ such that $\tran  \sigma(\gamma_r) J_r \gamma_r$ lies in the diagonal $F$-split torus of $\GL_r(E)$.
For instance if $r$ is even, we set
\begin{align*}
\gamma_r = \left( \begin{matrix} 
1 & & & &  &1 \\
   & \ddots & & & \iddots & \\
& & 1 & 1 & & \\
& & 1 & -1 & & \\
& \iddots & & & \ddots & \\
1 & & & & & -1
\end{matrix} \right)
\end{align*}
and if $r$ is odd, we take
\begin{align*}
\gamma_r = \left( \begin{matrix} 
1 & & && &  &1 \\
   & \ddots & && & \iddots & \\
& & 1 &0& 1 & & \\
& & 0& 1 & 0& & \\
& & 1 &0& -1 & & \\
& \iddots & && & \ddots & \\
1 & & & & & & -1
\end{matrix} \right).
\end{align*}
Define $\gamma = \gamma_n$ and notice that
\begin{align}
\tran \sigma(\gamma) w_\ell \gamma = \diag(\underbrace{2,\ldots,2}_{\floor{\frac{n}{2}}},\widehat 1, \underbrace{-2, \ldots, -2}_{\floor{\frac{n}{2}}}), 
\end{align}
lies in the diagonal $F$-split torus $A_T$ of $G$.  

Let $\mathbf T$ be the maximal (non-split) diagonal $F$-torus of $\G$.  The torus $\mathbf T$ is obtained by restriction of scalars of the diagonal torus of $\GL_n$. 
Let $T = \mathbf T(F)$, and identify $T$ with the diagonal matrices in $\GL_n(E)$. 
Let $A_T$ be the $F$-split component of $T$. 
Define $T_0 = {}^\gamma T$, then the $F$-split component  of $T_0$ is $A_0={}^\gamma A_T$.  
The tori $T$, $A_T$, $T_0$ and $A_0$ are all $\theta$-stable.
Observe that $A_0$ is a maximal $F$-split torus of $G$ that is $\theta$-split. 
In particular, $A_0$ is a maximal $(\theta, F)$-split torus of $G$.
Indeed, we have that ${}^t\sigma(\gamma) w_\ell \gamma $ lies in the abelian subgroup $A_T$; therefore, for any $\gamma t \gamma^{-1} \in A_0$, we have
\begin{align*}
\theta(\gamma t \gamma^{-1}) 
& = w_\ell^{-1} {}^t\sigma({\gamma})^{-1} ({}^t\sigma({t})^{-1}) {}^t\sigma({\gamma}) w_\ell \\
& = \gamma ({}^t\sigma(\gamma) w_\ell \gamma)^{-1} t^{-1} ({}^t\sigma(\gamma) w_\ell \gamma) \gamma^{-1} \\
& = (\gamma t \gamma^{-1})^{-1},
\end{align*}
where we've used that ${}^t\sigma({t})^{-1} = t^{-1}$, for any $t\in A_T$.

\begin{lem}\label{unitary-theta-split-comp}
For any $x\in X$, the $(\theta_x,F)$-split component of $G$, which we denote by $S_{G,x}$, is equal to the $F$-split component $A_G$ of the centre of $G$.  
\end{lem}

\begin{proof}
Let $z \in A_G$.  Since $z$ is a diagonal matrix with entries in $F^\times$, we have ${}^t \sigma(z) = z$; moreover, since $z$ is central in $G$,
\begin{align*}
\theta_x(z) &= x^{-1} {}^t \sigma(z)^{-1} x 
		       = x^{-1} z ^{-1} x
		       = z^{-1},
\end{align*}
It follows that  $S_{G, x} = (A_G)^\circ = A_G$ (\textit{cf.}~\Cref{sec-tori-involution}).
\end{proof}

Let $\Phi = \Phi(G, A_T)$ be the root system of $G$ with respect to $A_T$, with standard base $\Delta = \{ \epsilon_i-\epsilon_{i+1} : 1 \leq i \leq n-1\}$. 
Let $\Phi_0 = \Phi (G, A_0)$ be the root system of $G$ with respect to $A_0={}^\gamma A_T$. 
Observe that $\Phi_0 = {}^\gamma \Phi$.
Set $\Delta_0 = {}^\gamma \Delta$.  
The set of positive roots of $\Phi_0$ with respect to $\Delta_0$ is denoted $\Phi_0^+$.
We have that $\Phi_0^+ = {}^\gamma \Phi^+$, where $\Phi^+$ is the set of positive roots in $\Phi$ determined by $\Delta$.  Our current aim is to use $\Phi_0$ to determine the (standard) $\theta$-split parabolic subgroups of $G$.  First, we note the following.

\begin{lem}\label{lem-root-to-neg}
For any $\alpha \in \Phi_0$, we have $\theta(\alpha) = -\alpha$.  
\end{lem}

\begin{proof}
Let $\alpha \in \Phi_0$.  For any $a\in A_0$, we have that $\theta(a) = a^{-1}$; therefore, 
\begin{align*}
(\theta\alpha)(a) & = \alpha(\theta(a)) = \alpha(a^{-1})  = \alpha(a)^{-1} = (-\alpha)(a).
\end{align*}
Since $a \in A_0$ was arbitrary, we have that $\theta(\alpha) = -\alpha$.
\end{proof}

The following two corollaries of \Cref{lem-root-to-neg} follow immediately (\textit{cf.}~\Cref{defn-theta-base}).

\begin{cor}
The set $\Phi_0^\theta$ of $\theta$-fixed roots in $\Phi_0$ is empty.
\end{cor}

\begin{cor}\label{any-base-theta-base}
Any set of simple roots in $\Phi_0$ is a $\theta$-base for $\Phi_0$. In particular, $\Delta_0$ is a $\theta$-base.
\end{cor}

Explicitly, $\Delta_0 = \{{}^\gamma(\epsilon_i - \epsilon_{i+1}) : 1 \leq i \leq n-1\}$ and, by \Cref{any-base-theta-base}, the set of simple roots $\Delta_0$ is a $\theta$-base for $\Phi_0$.  Since the maximal $F$-split torus $A_0$ is a maximal $(\theta,F)$-split torus, the restricted root system of $H\backslash G$ is just the root system $\Phi_0$ of $G$. 
  The next proposition now follows immediately.

\begin{prop}\label{std-theta-split}
Every parabolic subgroup of $G$ standard with respect to $\Delta_0$ is a $\theta$-split parabolic subgroup.  
Any such parabolic subgroup is the $\gamma$-conjugate of the usual block-upper triangular parabolic subgroups of $G$.
\end{prop}

By \Cref{KT08-lem-2.5}(1), we have that any $\theta$-split parabolic subgroup of $G$ is $(\Hbf \mathbf T_0)(F)$-conjugate to a $\Delta_0$-standard $\theta$-split parabolic subgroup.

We now consider $\theta$-stable parabolic subgroups; in particular, we are concerned with determining which proper $\theta$-stable parabolic subgroups admit a $\theta$-elliptic Levi factor.

\begin{defn}
Let $(\underline n) = (n_1,\ldots, n_r)$ be a partition of $n$, we say that $(\underline n)$ is {balanced} if $n_i = n_{r+1-i}$, for $1 \leq i \leq r$.
\end{defn} 

Let $(\underline n) = (n_1,\ldots, n_r)$ be a partition of $n$.
The opposite partition to  $(\underline n)$ is  $(\underline n)^{\operatorname{op}}=(n_r, \ldots, n_1)$.  This terminology reflects that the standard upper-triangular parabolic subgroup that is $\GL_n$-conjugate to the (lower-triangular) opposite parabolic of $P_{(\underline n)}$ is precisely $P_{(\underline n)^{\operatorname{op}}}$.
Observe that $(\underline n)$ is balanced if and only if $(\underline n)^{\operatorname{op}} = (\underline n)$.

\begin{lem}\label{balanced-stable-pblc}
The $\theta$-stable block upper-triangular parabolic subgroups of $G$ correspond to balanced partitions of $n$.
The only such parabolic that has a $\theta$-stable $\theta$-elliptic Levi subgroup is $P_{(n/2,n/2)}$, in the case that $n$ is even.
\end{lem}

\begin{proof}
Let $(\underline n) = (n_1,\ldots, n_r)$ be a partition of $n$.  Let $A = A_{(\underline n)}$ be the diagonal $F$-split torus corresponding to $(\underline n)$.
The parabolic subgroup $P = P_{(\underline n)}$ is $\theta$-stable if and only if its standard Levi subgroup $M = C_G(A)$ and unipotent radical $N$ are $\theta$-stable; moreover, $M$ is $\theta$-stable if and only if $A$ is $\theta$-stable.
Observe that $A$ is $\theta$-stable if and only if $w_\ell \in N_G(A)$. Indeed, $\theta(a) = w_\ell^{-1} {}^t \sigma({a})^{-1} w_\ell$ and $A$ is stable under the involution $a \mapsto {}^t \sigma({a})^{-1} = {a}^{-1}$.
It is immediate that $\theta(A) = A_{(\underline n)^{\operatorname{op}}}$; moreover, 
$A$ is $\theta$-stable if and only if $(\underline n)^{\operatorname{op}} = (\underline n)$ if and only if $(\underline n)$ is balanced.
Therefore, it suffices to show that if $N$ is $\theta$-stable  if and only if $(\underline n)$ is balanced.  
First note that $N$ is stable under the map $n \mapsto \sigma({n})^{-1}$, on the other hand, the transpose map sends $N$ to the opposite unipotent radical $N^{\operatorname{op}}$.   
In particular, $N$ is $\theta$-stable if and only if $N = w_\ell^{-1} N^{\operatorname{op}} w_\ell$.
A simple matrix computation shows that this occurs if and only if  $(\underline n)^{\operatorname{op}} = (\underline n)$, i.e., $(\underline n)$ is balanced.

Let $(\underline n) = (n_1,\ldots, n_{\floor{r/2}},\widehat{n_\bullet}, n_{\floor{r/2}},\ldots,n_1)$ be a balanced partition of $n$.
Now, we show that $M$ is $\theta$-elliptic if and only if $(\underline n) = (n/2,n/2)$ by applying \cite[Lemma 3.8]{smith2018}, which states that a $\theta$-stable Levi subgroup $M$ is $\theta$-elliptic if and only if $S_M = S_G$.
An element $a = \diag(a_1,\ldots, a_n)$ of $A_T$ is $\theta$-split if and only if $a$ centralizes $w_\ell$. 
Indeed, since $A_T$ is pointwise fixed by taking the transpose-Galois conjugates, we have that
\begin{align*}
\theta(a) &= w_\ell^{-1} \diag(a_1^{-1},\ldots, a_n^{-1}) w_\ell 
	       = \diag(a_n^{-1},\ldots, a_1^{-1}), 
\end{align*}
which is equal to $a^{-1}$ if and only if $a_i = a_{n+1-i}$, for all $1\leq i \leq n$.  It follows that $a^{-1} = \theta(a)$ if and only if $a \in C_{A_T}(w_\ell)$, where
\begin{align*}
C_{A_T}(w_\ell) & = \{\diag(a_1,\ldots, a_{\floor{n/2}},\widehat a_\bullet, a_{\floor{n/2}},\ldots, a_1) : a_i \in F^\times, 1\leq i \leq \floor{n/2} \}.
\end{align*}
The $(\theta,F)$-split component $S_M$ of $M$ is thus equal to the identity component of $A \cap C_{A_T}(w_\ell)$.
We have that $A \cap C_{A_T}(w_\ell)$ is equal to the $F$-split torus
\begin{align}\label{eq-centralize-w-ell}
\left\{ \diag(\underbrace{a_1,..., a_1}_{n_1},..., \underbrace{a_{\floor{r/2}},..., a_{\floor{r/2}}}_{n_{\floor{r/2}}},\widehat{\underbrace{a_{\bullet},..., a_{\bullet}}_{n_\bullet}},\underbrace{a_{\floor{r/2}},..., a_{\floor{r/2}}}_{n_{\floor{r/2}}},...,\underbrace{a_1,..., a_1}_{n_1}) \right\};
\end{align}
in particular, $S_M = A \cap C_{A_T}(w_\ell)$.
By \Cref{unitary-theta-split-comp}, we have $S_G = A_G$ and observe that $S_M$ is equal to $A_G$ if and only if $r=2$, that is, $n$ is even and $(\underline n) = (n/2,n/2)$, as claimed.  
\end{proof}

\begin{rmk}
When $n$ is even, we set $L= M_{(n/2,n/2)}$ and reiterate that $L$ is the only proper block-diagonal $\theta$-elliptic Levi subgroup of $G$.
\end{rmk}

\begin{cor}\label{min-stable-pblc-unitary}
The minimal parabolic (Borel) subgroup $Q_0$ of $G$ consisting of the upper-triangular matrices is a $\theta$-stable minimal parabolic of $G$.  In particular, $\mathbf Q_0 = R_{E/F} {\bf B}$, where ${\bf B}$ is the upper-triangular Borel subgroup of $\GL_n$.
\end{cor}

\begin{proof}
The partition $(1,\ldots,1)$ is balanced; apply Lemma 4.15.
\end{proof}

\begin{cor}
The $F$-subgroup $Q_0 \cap H$ of $H$, consisting of the upper-triangular elements of $H$, is a Borel subgroup of $H$.
\end{cor}

\begin{proof}
See, for instance, \cite[Lemma 3.1]{gurevich--offen2016}.
\end{proof}

\begin{lem}
There are no proper $\theta$-elliptic Levi subgroups of $G$ that contain $A_0$.
\end{lem}

\begin{proof}
The maximal $F$-split torus $A_0$ of $G$ is $(\theta,F)$-split. The lemma follows from \cite[Lemma 3.8]{smith2018}.
\end{proof}

Recall that a parabolic subgroup $P$ of $G$ is called $A_T$-semi-standard if $P$ contains $A_T$. If $P$ is an $A_T$-semi-standard parabolic subgroup, then there is a unique Levi factor $M$ of $P$ that contains $A_T$. We refer to $M$ as the $A_T$-semi-standard Levi factor of $P$.

\begin{lem}\label{lem-A_T-std-theta-ell-Levi}
Let $n \geq 2$ be an integer.
\begin{enumerate}
\item If $n$ is odd, then there are no proper $\theta$-elliptic $A_T$-semi-standard Levi subgroups of $G = \GL_n(E)$.
\item If $n$ is even, then $L = M_{(n/2,n/2)}$ is the only maximal proper $\theta$-elliptic $A_T$-semi-standard Levi subgroup of $G=\GL_n(E)$, up to conjugacy by Weyl group elements $w \in W=W(G,A_T)$, such that $w^{-1}w_\ell w \in N_G(L) \setminus L$.
\end{enumerate}
\end{lem}

\begin{proof}
We give a sketch of the proof.
We identify the Weyl group $W = W(G,A_T)$ of $A_T$ in $G$ with the subgroup of permutation matrices in $G$.
By \Cref{balanced-stable-pblc}, if $n$ even, then $L = M_{(n/2,n/2)}$ is $\theta$-elliptic.

First, let $P= MN$ be a $\theta$-stable maximal proper $A_T$-semi-standard parabolic subgroup of $G$ with $\theta$-stable $A_T$-semi-standard Levi subgroup $M$.
It is well known that $P$ is $W$-conjugate to a unique standard (block upper-triangular) maximal parabolic subgroup of $G$.
In particular, $M = w M_{(n_1,n_2)}w^{-1}$ for some partition $(n_1,n_2)$ of $n$ and some $w\in W$.
Moreover, $M$ is $\theta$-stable if and only if its $F$-split component $A_M=wA_{(n_1,n_2)}w^{-1}$ is contained in a torus $A_{(\underline n)}$, for some balanced partition $(\underline n)$ of $n$.
Assume that $M$ is $\theta$-stable and let $(\underline n)$ be the coarsest partition such that $wA_{(n_1,n_2)}w^{-1}$ is contained in $A_{(\underline n)}$.
One may regard the $F$-split component $A_M = wA_{(n_1,n_2)}w^{-1}$ of $M$ as being obtained by a two-colouring of $(\underline n)$.
That is, regard $\diag(\underbrace{a,\ldots,a}_{n_1},\underbrace{b,\ldots,b}_{n_2}) \rightarrow A_{(\underline n)}$ as a two-colouring of the partition $(\underline n)$.
One may readily verify that:
\begin{enumerate}
\item If $n$ is odd, then, by considering \eqref{eq-centralize-w-ell}, we see that $wA_{(n_1,n_2)}w^{-1}$ contains at least a rank-one $F$-split torus, non-central in $G$, consisting of $\theta$-split elements.
(Indeed, since $n$ is odd, the central segment in \eqref{eq-centralize-w-ell} must appear.)
In particular, $M$ cannot be $\theta$-elliptic by \cite[Lemma 3.4]{smith2018}. 
Moreover, by \Cref{contain-theta-elliptic}, $M$ cannot contain a $\theta$-elliptic Levi subgroup of $G$.  It follows that there are no $A_T$-semi-standard $\theta$-elliptic Levi subgroups when $n$ is odd.
\item Suppose that $n$ is even.
In light of \eqref{eq-centralize-w-ell}, we see that $M$ is $\theta$-elliptic if and only if $(\underline n)$ is a refinement of the partition $(n/2,n/2)$ of $n$.  
In particular, it readily follows that $n_1 = n_2 = n/2$ and $M$ is conjugate to $L$.
\end{enumerate}
Observe that $\theta(w) = w_\ell^{-1} w w_\ell$, for any $w\in W$. It is straightforward to check that $M=wLw^{-1}$ is $\theta$-stable if and only if $w^{-1}w_\ell w \in N_G(L)$.
It can be verified that a $\theta$-stable conjugate $M=wLw^{-1}$ of $L$ is $\theta$-elliptic if and only if $w^{-1}w_\ell w \notin C_G(A_L) =  L$. Thus $L$ is the only maximal $A_T$-semi-standard $\theta$-elliptic Levi subgroup of $G$ (up to conjugacy).
\end{proof}

\begin{note}
Observe that $L = M_{(n/2,n/2)}$ does not contain any proper $\theta$-elliptic Levi subgroups. We argue by contradiction.
Suppose that $L' \subsetneq L$ is a $\theta$-elliptic Levi subgroup of $G$.  Notice that $L'$ is also $\theta$-elliptic in $L$.
Since $L'$ is proper in $L$, it follows from  \Cref{contain-theta-elliptic} that $L'$ is contained in a $\theta$-stable maximal proper Levi subgroup $L''$ of $L$.
Without loss of generality, $L'' \cong M_{(k_1,k_2)} \times \GL_{n/2}(F)$.  However, considering the action of $\theta$ on $L$ described in \eqref{theta-of-l}, we observe that no such Levi subgroup $L''$ can be $\theta$-stable.  
\end{note}

\Cref{lem-A_T-std-theta-ell-Levi} does not give a complete characterization of the $\theta$-elliptic Levi subgroups of $G$. The following lemma takes us closer to the desired result; in particular, up to $H$-conjugacy (and choice of standard parabolic) we have all of relevant $\theta$-stable parabolic subgroups.

\begin{lem}\label{semi-std-stbl-H-cong}
Let $\mathbf P$ be any $\theta$-stable parabolic subgroup of $\G$, then $P = \mathbf P(F)$ is $H$-conjugate to a $\theta$-stable $A_T$-semi-standard  parabolic subgroup.
\end{lem}

\begin{proof}
First, note that by \Cref{min-stable-pblc-unitary} and \cite[Lemma 3.5]{helminck--wang1993}, the torus $(\mathbf{A}_T \cap \Hbf)^\circ$  is a maximal $F$-split torus of $H$.
Let $\mathbf P = \mathbf M\mathbf N$ be a $\theta$-stable parabolic subgroup with the indicated Levi factorization, where $\mathbf M$ and $\mathbf N$ are both $\theta$-stable.  Let $\mathbf P_\bullet$ be a minimal $\theta$-stable parabolic subgroup of $\G$ contained in $\mathbf P$.  Let $\mathbf A_\bullet$ be a $\theta$-stable maximal $F$-split torus contained in $\mathbf P_\bullet$ \cite[Lemma 2.5]{helminck--wang1993}.
By \cite[Corollary 5.8]{helminck--wang1993}, there exists $g = nh \in \left(N_\G(\mathbf A_\bullet) \cap N_\G((\mathbf A_{\bullet}\cap \Hbf)^\circ)\right)(F) \Hbf(F)$ such that $g^{-1} \mathbf P_\bullet g = \mathbf Q_0$.  Note that $n$ normalizes $(\mathbf A_{\bullet}\cap \Hbf)^\circ$ and all of $\mathbf A_\bullet$, while $h$ is $\theta$-fixed.
Observe that
\begin{align}
g^{-1}A_\bullet g = h^{-1} n^{-1} A_\bullet n h = h^{-1} A_\bullet  h \subset Q_0;
\end{align}
in particular, $g^{-1}A_\bullet g$ is a $\theta$-stable maximal $F$-split torus.
Let $\mathbf U_0$ be the unipotent radical of $\mathbf Q_0$.
By \cite[Lemma 2.4]{helminck--wang1993}, $g^{-1}A_\bullet g$ is $(\Hbf \cap \mathbf U_0)(F)$-conjugate to $A_T$.
It follows that there exists $h' \in (\Hbf \cap \mathbf U_0)(F)$ such that
\begin{align}
A_T & = h'^{-1} g^{-1}A_\bullet g  h' = h'^{-1} h^{-1}A_\bullet h  h' =  (h  h' )^{-1}A_\bullet (h  h');
\end{align}
moreover, we have that $A_T  = (h  h' )^{-1}A_\bullet (h  h')$ is contained in $(h  h' )^{-1}\mathbf P (h  h')$ and $\mathbf P$ is $H$-conjugate to a $\theta$-stable $A_T$-semi-standard  parabolic subgroup.
\end{proof}

Let $n\geq 2$ be an even integer. Let $l = \diag(x,y) \in L = M_{(n/2,n/2)}$.  We compute that
\begin{align}\label{theta-of-l}
\theta(l) & = w_\ell^{-1}  \left (\tran \sigma({l})^{-1} \right)w_\ell 
  = \left( \begin{matrix} J_{n/2}^{-1} \tran \sigma({y})^{-1}J_{n/2} & 0 \\ 0 &  J_{n/2}^{-1} \tran \sigma({x})^{-1} J_{n/2} \end{matrix} \right) .
\end{align}
It follows immediately from \eqref{theta-of-l} that $l$ is $\theta$-fixed if and only if 
\begin{align*}
	y &= \theta_{J_{n/2}}(x) = J_{n/2}^{-1} \left( \tran  \sigma(x)^{-1} \right) J_{n/2},
\end{align*}
and observes that, 
\begin{align}\label{theta-fixed-L}
L^\theta = \left \{ \left( \begin{matrix} x & 0 \\ 0 &  \theta_{J_{n/2}}(x) \end{matrix} \right) : x \in \GL_{n/2}(E) \right\} \cong \GL_{n/2}(E).
\end{align}
From \eqref{theta-fixed-L}, we immediately obtain a characterization of the $\theta$-fixed points of the associate Levi subgroup $M={}^\gamma L$ of the $\Delta_0$-standard parabolic subgroup $P = {}^\gamma Q$.

\begin{lem}\label{split-levi-fixed-pts}
The Levi subgroup $M = {}^\gamma L$ is the $\theta$-stable Levi subgroup of a standard $\theta$-split parabolic $P=MN = {}^\gamma P_{(n/2,n/2)}$.  The $\theta$-fixed points of $M$ are isomorphic to a product of two copies of the unitary group $\mathbf U_{E/F, 1_{n/2}} = \{ x \in \GL_{n/2}(E) : x^{-1} = \tran  \sigma(x)\}$, where $1_{n/2}$ is the $n/2\times n/2$ identity matrix.  
\end{lem}
\begin{proof}
Let ${}^\gamma m \in M$, where $m \in L$.  
We have that 
\begin{align*}
\theta({}^\gamma m) 
& =  w_\ell^{-1}\ \tran  \sigma(\gamma)^{-1}\ \tran  \sigma(m)^{-1}\ \tran  \sigma( \gamma) \ w_\ell \\
 &=  \gamma (\tran \sigma(\gamma) w_\ell \gamma )^{-1} \ \tran  \sigma({m})^{-1}\ (\tran \sigma(\gamma) w_\ell \gamma ) \gamma^{-1} \\
 &= \gamma\ \tran  \sigma({m})^{-1}  \gamma^{-1},
\end{align*}
where the last equality 
holds since $\tran \sigma(\gamma) w_\ell \gamma \in A_L$ centralizes $m\in L$.
It follows that ${}^\gamma m =  \theta({}^\gamma m)$ if and only if $m =   \tran  \sigma({m})^{-1}$. 
Writing $m$ as a block-diagonal matrix $m =\diag(x,y)$,  
we have $m = \tran  \sigma({m})^{-1}$ if and only if $x = \tran  \sigma({x})^{-1}$ and $y = \tran  \sigma({y})^{-1}$.  It follows that
\begin{align}\label{theta-fixed-M-split-Levi}
M^\theta =  \left\{ \gamma \left( \begin{matrix} x & 0 \\ 0 &  y \end{matrix} \right)\gamma^{-1} : x,y \in \GL_{n/2}(E), x = \tran  \sigma({x})^{-1}, y = \tran  \sigma({y})^{-1} \right \},
\end{align}
and $M^\theta \cong \mathbf U_{E/F, 1_{n/2}} \times \mathbf U_{E/F, 1_{n/2}}$, as claimed.
\end{proof}

In Lemma \ref{unitary-gen-split-fixed-pts}, we will determine the $\theta$-fixed points of the Levi subgroup of an arbitrary maximal $\theta$-split parabolic subgroup that has $\theta$-stable Levi factor associate to $L$.

\section{Relative discrete series for $\mathbf{U}_{E/F}(F) \backslash \GL_{2n}(E)$}\label{sec-rds}
From now on, let $G = \GL_{2n}(E)$, where $n\geq 2$, and let $H = \mathbf U_{E/F,w_\ell}(F)$ be the $\theta$-fixed points of $G$.
Recall that $H$ is (the $F$-points of) a quasi-split unitary group.
Let $Q=P_{(n,n)}$ be the type $(n,n)$ block-upper triangular parabolic subgroup of $G$, with block-diagonal Levi factor $L=M_{(n,n)}$ and unipotent radical $U_{(n,n)}$.

In this section we prove \Cref{thm-unitary-rds}, the main result of the paper.  We construct representations in the discrete series of $H\backslash G$ via parabolic induction from $L^\theta$-distinguished discrete series representations of $L$.

The parabolic subgroup $Q$ is conjugate to the $\Delta_0$-standard maximal $\theta$-split parabolic $P_\Omega$, where
$\Omega = \Delta_0 \setminus \{ {}^\gamma(\epsilon_{n} - \epsilon_{n+1}) \}$. 
In particular, $Q =\gamma^{-1} P_\Omega \gamma$,  $L = \gamma^{-1} M_\Omega \gamma$ and $U = \gamma^{-1} N_\Omega \gamma$.
For any $\Theta \subset \Delta_0$, we use the representatives $[W_\Theta \backslash W_0 / W_\Omega]$ for the double-coset space $P_\Theta \backslash G / P_\Omega$ (\textit{cf.}~\Cref{lem-nice-reps}).
There is an isomorphism $P_\Theta \backslash G / P_\Omega \cong P_\Theta \backslash G / Q$ given by $w \mapsto w \gamma$. 
\subsection{A few more ingredients}
Here, we assemble the remaining representation theoretic results needed to state and prove \Cref{thm-unitary-rds}.
\subsubsection{The inducing discrete series representations}\label{sec-inducing-ds}
The irreducible representations distinguished by arbitrary unitary groups are characterized in the paper \cite{feigon--lapid--offen2012}, continuing the work of Jacquet \textit{et al.}, for instance in \cite{jacquet--lai1985,jacquet2001,jacquet2010}.
Feigon, Lapid and Offen study both local and global distinction, largely using global methods.
In particular, they show that an irreducible square integrable representation $\pi$ of $G$ is $H^x$-distinguished if and only if $\pi$ is Galois invariant.
Although Feigon, Lapid and Offen prove much stronger results, we'll recall only what we need for our application. 
The following appears as \cite[Corollary 13.5]{feigon--lapid--offen2012}. 
Recall from \eqref{Hermitian} that $X$ is the variety of Hermitian matrices in $\GL_n(E)$.
\begin{thm}[Feigon--Lapid--Offen]\label{FLO-cor-13-5}
Let $\pi$ be an irreducible admissible essentially square integrable representation of $\GL_n(E)$.
For any $x\in X$, the following conditions are equivalent:
\begin{enumerate}
\item the representation $\pi$ is Galois invariant, that is $\pi \cong {}^\sigma \pi$.
\item the representation $\pi$ is $H^x$-distinguished.
\end{enumerate}
In addition, $\dim \Hom_{H^x}(\pi,1) \leq 1$.
\end{thm}
The multiplicity-one statement appears as \cite[Proposition 13.3]{feigon--lapid--offen2012}. It is known that local multiplicity-one for unitary groups does not hold in general, see \cite[Corollary 13.16]{feigon--lapid--offen2012}  for instance, which gives a lower bound for the dimension of $\Hom_{H^x}(\pi,1)$ for Galois invariant generic representations $\pi$.
On the other hand, Feigon, Lapid and Offen are able to extend \Cref{FLO-cor-13-5} to all {ladder representations} (\textit{cf.}~\cite[Theorem 13.11]{feigon--lapid--offen2012}).

\subsubsection{Distinction of inducing representations}
\begin{prop}\label{classify-L-theta-dist}
An irreducible admissible representation $\pi_1\otimes\pi_2$ of $L$ is $L^\theta$-distinguished if and only if $\pi_2$ is equivalent to the Galois-twist of $\pi_1$, that is, if and only if $\pi_2 \cong {}^\sigma \pi_1$.
\end{prop}

We actually prove a slightly more general result from which Proposition \ref{classify-L-theta-dist} is a trivial corollary, by taking into account the description of $L^\theta$ given in \eqref{theta-fixed-L}.

\begin{lem}\label{ind-data-dist}
Let $m \geq 1$ be an integer.
Let $G' = \GL_m(E) \times \GL_m(E)$, $x\in \GL_m(E)$ a Hermitian matrix, and define
\begin{align*}
H' =  \left \{ \left( \begin{matrix} A & 0 \\ 0 &  \theta_{x}(A) \end{matrix} \right) : A \in \GL_{m}(E) \right\}.
\end{align*}
An irreducible admissible representation $\pi_1\otimes \pi_2$ of $G'$ is $H'$-distinguished if and only if $\pi_2$ is equivalent to the Galois-twist of $\pi_1$, i.e.,  $\pi_2 \cong {}^\sigma \pi_1$.
\end{lem}

\begin{proof}
First, note that a representation $\pi_1\otimes \pi_2$ is $H'$-distinguished if and only if $\pi_2$ is equivalent to ${}^{\theta_x} \wt \pi_1$, the $\theta_x$-twist of the contragredient of $\pi_1$.
It suffices to show that for any Hermitian matrix $x$ in $\GL_m(E)$, and any irreducible admissible representation $\pi$ of $\GL_m(E)$, the Galois-twisted representation ${}^\sigma\pi$ is equivalent to ${}^{\theta_x}\widetilde\pi$.
By a result of Gel'fand and Kazhdan \cite[Theorem 2]{gelfand--kazhdan1975}, we have that $\widetilde \pi$ is equivalent to $\widehat \pi$, where the representation $\widehat\pi$, is defined by $\widehat\pi(g) = \pi({}^t g^{-1})$ acting on the space $V$ of $\pi$.
Since $\pi$ is admissible, we have $\wt{\wt \pi} \cong \pi$;  
thus, we see that $\widehat{(\widetilde \pi)} \cong \pi$.
On the other hand, the representation ${}^{\theta_x}\widetilde\pi$ on $\widetilde {V}$ is given by ${}^{\theta_x}\widetilde\pi (g)  = \widetilde\pi (\theta_x(g))$.
Using that $x$ is Hermitian, it is readily verified that
\begin{align*}
{}^{\theta_x}\widetilde\pi (g) 
 &= {}^\sigma \widehat{(\widetilde\pi)} (x g x^{-1}) 
 =  {}^{x^{-1}}({}^\sigma \widehat{(\widetilde\pi)}) (g).
\end{align*}
We observe that ${}^{\theta_x}\widetilde\pi$ is equivalent to ${}^\sigma \widehat{(\widetilde\pi)}$ since $\operatorname{Int} x^{-1}$ is an inner automorphism of $\GL_m(E)$. 
It is also clear that taking Galois twists commutes with the map sending $\pi$ to $\widehat  \pi$ (and twisting by $\operatorname{Int} x^{-1}$, since $x$ is Hermitian).  
Finally, we have shown that 
\begin{align*}
{}^{\theta_x}\widetilde\pi 
\cong {}^{x^{-1}}({}^\sigma \widehat{(\widetilde\pi)}) 
\cong {}^\sigma \widehat{(\widetilde\pi)} 
\cong {}^\sigma \pi,
\end{align*}
as claimed.
\end{proof}

\subsubsection{$H$-distinction of an induced representation}
Let $\tau'$ be an irreducible admissible representation of $\GL_{n}(E)$ and define $\tau = \tau' \otimes {}^\sigma \tau'$.  
By \Cref{classify-L-theta-dist}, the irreducible admissible representation $\tau$ of $L$ is $L^\theta$-distinguished.  
Let $\lambda$ be a nonzero element of $\Hom_{L^\theta}(\tau,1)$. 
The invariant form $\lambda$ is defined using the pairing of $\tau'$ with its contragredient.  
By \cite[Proposition 4.3.2]{lapid--rogawski2003}, 
we have that $\delta_Q^{1/2}$ restricted to $L^\theta$ is equal to $\delta_{Q\cap H} = \delta_{Q^\theta}$. 
By \Cref{cor-hom-injects}, we have the following result.

\begin{prop}\label{distinction-pi}
Let $\tau'$ be an irreducible admissible representation of $\GL_{n}(E)$.
If $\tau = \tau' \otimes {}^\sigma \tau'$,
then the induced representation $\pi = \iota_Q^G \tau$  is $H$-distinguished.
\end{prop}

\subsection{Computing exponents and distinction of Jacquet modules}\label{compute-exponents}
Let $P=MN$ be a proper $\theta$-split parabolic subgroup of $G$, with $\theta$-stable Levi factor $M$ and unipotent radical $N$.
By the Geometric Lemma (\Cref{geom-lem}) and \Cref{exp-irred-subq}, the exponents of $\pi = \iota_Q^G \tau$ along $P$ are given by
\begin{align*}
\Exp_{A_M}(\pi_N) = \bigcup_{y\in M\backslash S(M,L) / L} \Exp_{A_M}(\mathcal F_N^y(\tau)),
\end{align*}
and
the exponents $\Exp_{A_M}(\mathcal F_N^y(\tau))$ are the central characters of the irreducible subquotients of the representations $\mathcal F_N^y(\tau)$ (\textit{cf.}~\Cref{rmk-geom-lemma-notation}).  
Note that restriction of characters from $A_M$ to the $(\theta,F)$-split component $S_M$ provides a surjection from $\Exp_{A_M}(\pi_N)$ to $\Exp_{S_M}(\pi_N)$, for instance see \cite[Lemma 4.15]{smith2018}.

Let $y\in M\backslash S(M,L) / L$.
There are two situations that we need to consider: 
\begin{itemize}
\item[]\textit{Case (1):} when $P \cap {}^y L = {}^y L$, and
\item[]\textit{Case (2):} when $P \cap {}^y L \subsetneq {}^yL$ is a proper parabolic subgroup of ${}^yL$.
\end{itemize}

By \Cref{KT08-lem-2.5}, there exists a $\theta$-split subset $\Theta \subset \Delta_0$ and an element $g \in (\mathbf{T_0}\Hbf)(F)$ such that $P = g P_\Theta g^{-1}$.
We choose a representative of $y$ with the form $y = gw\gamma$, where $w \in [W_\Theta \backslash W_0 / W_\Omega]$. Recall that $Q = \gamma^{-1} P_\Omega \gamma$, $U = \gamma^{-1} N_\Omega \gamma$, and $L = \gamma^{-1} M_\Omega \gamma$, where $\Omega = \Delta_0 \setminus \{{}^\gamma(\epsilon_n - \epsilon_{n+1})\}$.
\subsubsection{Case (1)}\label{sec-case-1}
Suppose that $P \cap {}^y L = {}^y L$. Then $M \cap {}^yL ={}^yL \cong M_{(n,n)}$. Moreover, the Levi subgroup $M$ must be a maximal proper Levi subgroup of $G$ that is associate to $L$.  It follows that, in this case, $\Theta = \Omega$.

There are exactly two elements $w \in [W_\Omega \backslash W_0 / W_\Omega]$ such that $y = gw\gamma$ satisfies $M \cap {}^yL={}^yL$: the identity and ${}^\gamma w_L$, where 
\begin{align*}
w_L &= \left( \begin{matrix} 0 & 1_{n} \\  1_{n} & 0 \end{matrix} \right) \in N_G(L),
\end{align*} 
and $1_{n}$ is the $n\times n$ identity matrix.
It follows that, in Case (1), there are two elements $y$ that we need to consider: $y = ge\gamma = g\gamma$ and $y = g{}^\gamma w_L \gamma = g\gamma w_L$.

\begin{note}
For a representation $\tau'$ of $\GL_{n}(E)$, we have ${}^{w_L}(\tau' \otimes {}^\sigma\tau') \cong {}^\sigma\tau' \otimes \tau'$.\footnote{The Weyl group element $w_L$ has the property that ${}^{w_L}Q = Q^{\operatorname{op}}$.} 	
\end{note}

We obtain the following.

\begin{lem}\label{n-even-max-pi-N}
Let $\tau = \tau' \otimes {}^\sigma \tau'$ be an irreducible admissible representation of $L$.
Let $P = {}^gP_\Omega$, $g \in (\mathbf{T_0}\Hbf)(F)$, be a maximal $\theta$-split parabolic subgroup with Levi associate to $L$.
Let $\pi = \iota_Q^G \tau$. 
\begin{enumerate}
\item If $y = g\gamma$, then $ \mathcal F_N^y(\tau) = {}^{g\gamma} \tau ={}^{g\gamma} \left( \tau' \otimes {}^\sigma \tau' \right)$.
\item If $y = g\gamma w_L$, then $\mathcal F_N^y(\tau) = {}^{g\gamma w_L} \tau = {}^{g\gamma} \left({}^\sigma \tau' \otimes \tau'\right) $
\end{enumerate}
\end{lem}

\begin{proof}
Indeed, for $y = g\gamma x$, where $x$ normalizes $L$, we have 
\begin{align*}
M \cap {}^y Q = M \cap g\gamma Q \gamma^{-1}g^{-1} = M \cap P = M
\end{align*}
 and 
 \begin{align*}
 P \cap {}^y L = P \cap g\gamma L \gamma^{-1}g^{-1} = P \cap M = M = {}^y L,
 \end{align*}
so we have that
 \begin{align*}
 \mathcal F^y_N (\tau)  = \iota_{M}^M \left(({}^y \tau)_{\{e\}}\right) = {}^y \tau.
 \end{align*}
\end{proof}

If $\tau$ is an irreducible unitary (e.g., a discrete series) representation of $L$, then the two subquotients $\mathcal F_N^{g\gamma}(\tau) = {}^{g\gamma} \left( \tau' \otimes {}^\sigma \tau' \right)$ and $\mathcal F_N^{g\gamma w_L}(\tau)={}^{g\gamma} \left({}^\sigma \tau' \otimes \tau'\right)$ of $\pi_N$ are irreducible and unitary.  
\subsubsection{Case (2)}\label{sec-case-2}
Suppose that $P \cap {}^y L \subsetneq {}^yL$ is a proper parabolic subgroup of ${}^yL$.
In particular, the Levi subgroup $M \cap {}^yL$ of $P \cap {}^y L$ is properly contained in ${}^yL$.
The following is the direct analogue of \cite[Proposition 8.5]{smith2018}.  The idea of the proof is to realize the exponents of $\pi$ along $P$ as restrictions of the exponents of $\tau$ along $P \cap {}^y L$ (\textit{cf.}~\cite[Lemma 4.16]{smith2018}) and to apply Casselman's Criterion to the discrete series $\tau$.
\begin{prop}\label{unitary-case2-exp-rel-cass}
Let $P$ be a maximal $\theta$-split parabolic subgroup of $G$ and $y\in P\backslash G / Q$ such that $P \cap {}^yL$ is a proper parabolic subgroup of ${}^yL$.
Let $\tau$ be a discrete series representation of $L$.
The exponents of $\pi=\iota_Q^G \tau$ along $P$ contributed by the subquotient $\mathcal F_N^y(\tau) = \iota_{M \cap {}^y Q}^M \left( {}^y \tau \right)_{N\cap {}^y L}$ of $\pi_N$ satisfy
the condition in \eqref{rel-casselman}.
\end{prop}

\begin{proof}
If $\tau$ is a discrete series representation of $L$, then $\tau$ satisfies Casselman's Criterion (\textit{cf.}~\cite[Theorem 6.5.1]{Casselman-book}).
In light of \Cref{technical-torus-nbhd-Unitary-general}, the argument that the exponents of $\pi=\iota_Q^G\tau$ along $P$ satisfy the Relative Casselman's Criterion (\Cref{rel-casselman-crit}) follows exactly as in the proof of \cite[Proposition 8.5]{smith2018}.
\end{proof}

\subsubsection{Distinction of $\mathcal F_N^y(\tau)$ and $\pi_N$}\label{sec-dist-jacquet}
A consequence of \Cref{unitary-case2-exp-rel-cass} is that we need only consider distinction of (subquotients of) the Jacquet module of $\pi = \iota_Q^G\tau$ in \textit{Case (1)} (\textit{cf.}~\Cref{rel-casselman-crit} and \Cref{non-dist-gen-eig-sp}).
We now determine the $\theta$-fixed points of $M$ and study $M^\theta$-distinction of the irreducible subquotients of $\pi_N$ in the case that $P \cap {}^y L = {}^y L$. 

\begin{lem}\label{norm-op-L-equal-L}
The intersection $N_G(L) \cap Q^{\operatorname{op}}$ of the normalizer of $L$ in $G$ with the opposite parabolic $Q^{\operatorname{op}}$ of $Q$
 is equal to $L$.
\end{lem}

\begin{proof}
Suppose that $q = \left( \begin{matrix} A & 0 \\ B & C \end{matrix} \right) \in N_G(L) \cap Q^{op}$, and let $l = \diag(l_1,l_2)$ be an arbitrary element of $L$. 
We have that
\begin{align*}
q l q^{-1}
 & = \left( \begin{matrix} A l_1 A^{-1} & 0 \\ B l_1 A^{-1} - C l_2 C^{-1}BA^{-1} & C l_2 C^{-1} \end{matrix} \right) \in L,
\end{align*}
and we see that
\begin{align*}
	B l_1 A^{-1} - C l_2 C^{-1}BA^{-1} &= 0,
\end{align*}
for all $l_1, l_2 \in \GL_{n}(E)$.
 This occurs if and only if $B = C l_2 C^{-1} B l_1^{-1}$ for all $l_1, l_2 \in \GL_{n}(E)$.  
If we take $l_2 = 1_{n}$ to be the identity, then $B l_1 = B$ for any $l_1 \in \GL_{n}(E)$, which occurs if and only if $B(l_1 - 1_{n}) = 0$ for any $l_1 \in \GL_{n}(E)$. 
In particular, since there exists $l_1 \in \GL_{n}(E)$ such that $l_1 - 1_{n}$ is invertible, we must have that $B= 0$.  The lemma follows.  
\end{proof}

\begin{lem}\label{unitary-gen-split-fixed-pts}
Let $M = g \gamma M_{(n,n)}\gamma^{-1}g^{-1}$, where $g \in (\Hbf \mathbf T_0)(F)$.
\begin{enumerate}
\item The subgroup $M^\theta$ of $\theta$-fixed points in $M$ is the $g\gamma$-conjugate of $L^{g\gamma\cdot\theta}$. 
\item We have that $L^{g\gamma\cdot\theta} = L^{\theta_{x_L}}$ is equal to the product $\mathbf U_{E/F, x_1} \times \mathbf U_{E/F, x_2}$ of unitary groups.
\item Explicitly,  $M^\theta = g\gamma \left( \mathbf U_{E/F, x_1} \times \mathbf U_{E/F, x_2} \right)\gamma^{-1} g^{-1}$ is isomorphic to a product of unitary groups.  
\item Let $\tau$ be an irreducible admissible representation of $L$.  Then ${}^{g\gamma} \tau$ is $M^\theta$-distinguished if and only if $\tau$ is $\mathbf U_{E/F, x_1} \times \mathbf U_{E/F, x_2}$-distinguished.
\end{enumerate}
\end{lem}

\begin{proof}
By Lemma \ref{norm-op-L-equal-L}, we have that $x_L = \gamma^{-1} g^{-1}\theta(g)\gamma = \diag(x_1, x_2) \in L$; moreover, $x_L$ is Hermitian.
Indeed, since ${}^t \gamma = \gamma = \sigma(\gamma)$ and $w_\ell = {}^t w_\ell = \sigma( w_\ell)$, we have
\begin{align*}
{}^t \sigma( x_L )& = {}^t\sigma({\gamma^{-1}g^{-1}\theta(g)\gamma}) \\
& = {}^t\sigma({\gamma^{-1}g^{-1}w_\ell^{-1} ({}^t \sigma( g) ^{-1}) w_\ell \gamma}) \\
& = \sigma(  {}^t\gamma)	\sigma( {}^tw_\ell )  g ^{-1}	\sigma({}^t w_\ell)^{-1}	\sigma({}^tg)^{-1}	\sigma({}^t \gamma)^{-1}\\
& = \gamma w_\ell g^{-1} w_\ell^{-1} {}^t\sigma(g)^{-1} \gamma^{-1} \\
& =  \gamma w_\ell g^{-1} w_\ell^{-1} {}^t\sigma(g)^{-1} w_\ell w_\ell^{-1} \gamma^{-1} \\
& = \gamma w_\ell g^{-1} \theta(g) w_\ell^{-1} \gamma^{-1} \\
& = z \gamma^{-1} g^{-1} \theta(g) \gamma z^{-1} \\
& = z x_L z^{-1} \\
& = x_L,
\end{align*}
where $z = {}^t\sigma(\gamma) w_\ell \gamma = \gamma w_\ell \gamma \in A_L$.
In fact, we have that $x_1$ and $x_2$ are Hermitian elements of $\GL_{n}(E)$.

Upon restriction to $L$, $g\gamma \cdot \theta = x_L \cdot \theta_e = \Int x_L^{-1} \circ  {}^t \sigma{( \ )}^{-1}$. 
Note also that $x_L \cdot \theta_e  = \theta_{e\cdot x_L} = \theta_{x_L}$.
 In particular, $l \in L$ is $g\gamma \cdot \theta$-fixed if and only if $l$ is $\theta_{x_L}$-fixed.
 Explicitly, $l \in L$ is $\theta_{x_L}$-fixed if and only if $l = x^{-1} {}^t\sigma( l )^{-1}x_L$.
Since $x_L$ is Hermitian, we have that $L^{g\gamma\cdot\theta} = L^{\theta_{x_L}}$ is equal to the product $\mathbf U_{E/F, x_1} \times \mathbf U_{E/F, x_2}$ of unitary groups.
Now, observe that $M^\theta = g\gamma L^{g\gamma\cdot\theta} (g\gamma)^{-1}$.  
 Indeed, if $m = g\gamma l \gamma^{-1}g^{-1} \in M$, where $l \in L$, then 
$m$ is $\theta$-fixed if and only if $l = (g\gamma\cdot \theta)(l)$ is $(g\gamma\cdot \theta)$-fixed.
\end{proof}

It is interesting to note that, even though we're interested in distinction by the quasi-split unitary group $\Hbf=\mathbf U_{E/F, w_\ell}$, we will need to consider distinction by (possibly non-quasi-split) unitary groups for Jacquet modules.

Define a representation $\rho$ of a Levi subgroup $M$ of $G$ to be regular if for every non-trivial element $w\in N_G(M)/M$ we have that the twist ${}^w \rho = \rho ( w^{-1} (\cdot) w)$ is not equivalent to $\rho$.
A representation $\pi_1\otimes\pi_2$ of $L$ is regular if and only if $\pi_1\neq \pi_2$.
The next lemma is \cite[Lemma 8.2]{smith2018}.
\begin{lem}\label{unitary-exponents}
Assume that $\tau$ is a regular unitary irreducible admissible representation of $L$.  
Let $P=MN$ be a $\theta$-split parabolic subgroup with Levi associate to $L$.
If $y \in M\backslash S(M,L) / L$ is such that $P\cap {}^yL = {}^yL$, then $\mathcal F_N^y(\tau)$ is irreducible and the central character $\chi_{N,y}$ of $\mathcal F_N^y(\tau)$ is unitary.
\end{lem}

From \Cref{n-even-max-pi-N}, it follows that $M^\theta$-distinction of ${}^{g\gamma} \tau$ (respectively, ${}^{g\gamma w_L} \tau$) is equivalent to $\mathbf U_{E/F, x_1}$-distinction of $\tau'$ and $\mathbf U_{E/F, x_2}$-distinction of ${}^\sigma \tau'$ (respectively, $\mathbf U_{E/F, x_1}$-distinction of ${}^\sigma \tau'$ and $\mathbf U_{E/F, x_2}$-distinction of  $\tau'$).
If $\tau = \tau' \otimes {}^\sigma \tau'$ is a regular discrete series representation, then $\tau' \ncong {}^\sigma\tau'$. It follows from \Cref{FLO-cor-13-5} that neither $\tau'$ nor ${}^\sigma \tau'$ can be distinguished by any unitary group. 
If $P=MN$ is any $\theta$-split parabolic such that $M$ is associate to $L$, then by \Cref{unitary-gen-split-fixed-pts}, neither of the irreducible subquotients of $\pi_N$,  described in \Cref{n-even-max-pi-N}, can be $M^\theta$-distinguished. 
We records this as the following.
\begin{cor}\label{no-dist-regularity}
Let $\pi = \iota_Q^G \tau$, where $\tau = \tau' \otimes {}^\sigma \tau'$ is a discrete series representation such that  $\tau'\ncong {}^\sigma \tau'$.  
Let  $P={}^gP_\Omega$, where $g\in (\Hbf \mathbf T_0)(F)$, be any maximal $\theta$-split parabolic subgroup with $\theta$-stable Levi $M=P\cap\theta(P)$ associate to $L$. 
Neither of the two irreducible unitary subquotients of $\pi_N$, twists of $\tau' \otimes {}^\sigma \tau'$ and ${}^\sigma\tau' \otimes \tau'$ (\textit{cf.}~\Cref{n-even-max-pi-N}), can be $M^\theta$-distinguished.
\end{cor}

\subsection{Constructing relative discrete series}
The following theorem is our main result.  The argument is the same as the proof of \cite[Theorem 6.3]{smith-phd2017}.  
\begin{thm}\label{thm-unitary-rds}
Let $Q = P_{(n,n)}$ be the upper-triangular parabolic subgroup of $G$ with standard Levi factor $L= M_{(n,n)}$ and unipotent radical $U = N_{(n,n)}$.
Let $\pi = \iota_Q^G \tau$, where $\tau = \tau' \otimes {}^\sigma \tau'$, and $\tau'$ is a discrete series representation of $\GL_{n}(E)$ such that $\tau'$ is not Galois invariant, i.e.,  $\tau'\ncong {}^\sigma \tau'$.
The representation $\pi$ is a relative discrete series representation for $H\backslash G$ that does not occur in the discrete series of $G$.
\end{thm}

\begin{proof}
Since $\tau$ is unitary and regular, by result of Bruhat \cite{bruhat1961} (\textit{cf.}~\cite[Theorem 6.6.1]{Casselman-book}), $\pi$ is irreducible.
In addition, $\pi$ is $H$-distinguished by \Cref{distinction-pi}.  The representation $\pi$ does not occur in the discrete series of $G$ by Zelevinsky's classification \cite{zelevinsky1980}.  Let $\lambda$ denote a fixed $H$-invariant linear form on $\pi$.
 It suffices to show that $\pi$ satisfies the Relative Casselman's Criterion \Cref{rel-casselman-crit}.
Let $P=MN$ be a proper $\theta$-split parabolic subgroup of $G$.
The exponents of $\pi$ along $P$ are the central characters of the irreducible subquotients of the representations $\mathcal F^y_N(\tau)$ given by the Geometric Lemma \ref{geom-lem} 
(see \Cref{compute-exponents}).
By \cite[Lemma 4.6]{kato--takano2010} and \Cref{unitary-case2-exp-rel-cass}, 
the condition \eqref{rel-casselman} is satisfied when $P \cap {}^yL$ is a proper parabolic subgroup of ${}^y L$. 
As in \Cref{unitary-exponents}, the only unitary exponents of $\pi$ along $P$ occur when $P \cap {}^yL = {}^yL$.
By \Cref{no-dist-regularity}, the only irreducible unitary subquotients of $\pi_N$ cannot be $M^\theta$-distinguished when $P \cap {}^y L = {}^y L$.
In the latter case, by \Cref{non-dist-gen-eig-sp}, the unitary exponents of $\pi$ along $P$ do not contribute to $\Exp_{S_M}(\pi_N,\lambda_N)$.
Therefore, \eqref{rel-casselman} is satisfied for every proper $\theta$-split parabolic subgroup of $G$.  Finally, by \Cref{rel-casselman-crit} the representation $\pi$ appears in the discrete spectrum of $H \backslash G$.
In particular, $\pi$ is $(H,\lambda)$-relatively square integrable for all nonzero $\lambda \in \Hom_H(\pi,1)$. 
\end{proof}

In addition, we note the following existence results.
First, we recall the structure of the representations in the discrete spectrum of $\GL_n(E)$.
Let $\rho$ be an irreducible unitary supercuspidal representation of $\GL_r(E)$, $r\geq 1$.  For an integer $k\geq 2$, write $\st(k,\rho)$ for the unique irreducible (unitary) quotient of the parabolically induced representation 
\begin{align*}
\iota_{P_{(r,\ldots,r)}}^{\GL_{kr}(E)}\left(\nu^{\frac{1-k}{2}} \rho \otimes \nu^{\frac{3-k}{2}} \rho \otimes \ldots \otimes \nu^{\frac{k-1}{2}} \rho\right)
\end{align*}
 of $\GL_{kr}(E)$ (\textit{cf.}~\cite[Proposition 2.10, $\S$9.1]{zelevinsky1980}), where $\nu(g) = |\det(g)|_E$, for any $g \in \GL_r(E)$. 
The representations $\st(k,\rho)$ are the {generalized Steinberg representations} and they are precisely the nonsupercuspidal discrete series representations of $\GL_{kr}(E)$ \cite[Theorem 9.3]{zelevinsky1980}. 
The usual Steinberg representation $\st_n$ of $\GL_n(E)$ is obtained as $\st(n,1)$.

\begin{prop}\label{prop-exist-reg-dist-ds-unitary}
Let $n\geq 2$ be an integer.
There exist infinitely many equivalence classes of non-supercuspidal discrete series representations $\tau$ of $\GL_n(E)$ that are not Galois invariant.
\end{prop}

Before giving a proof of \Cref{prop-exist-reg-dist-ds-unitary}, we note the following results.

\begin{cor}\label{cor-infinitely-many}
Let $n\geq 4$ be an integer. There exist infinitely many equivalence classes of RDS representations of the form constructed in \Cref{thm-unitary-rds} and such that  the discrete series $\tau$ is not supercuspidal.
\end{cor}

\begin{proof}
Apply \Cref{prop-exist-reg-dist-ds-unitary} and \cite[Theorem 9.7(b)]{zelevinsky1980}.
\end{proof}

\begin{prop}\label{prop-gen-st-unitary}
Let $\rho$ be an irreducible supercuspidal representation of $\GL_r(E)$, $r\geq 1$.
For $k\geq 2$, the generalized Steinberg representation $\st(k,\rho)$ of $\GL_{kr}(E)$ is Galois invariant if and only if $\rho$ is Galois invariant.
\end{prop}

\begin{proof}
First, observe that the twisted representation ${}^\sigma \st(k,\rho)$ is equivalent to the generalized Steinberg representation $\st(k,{}^\sigma\rho)$. 
It follows that $\st(k,\rho)$ is Galois invariant if and only if $\st(k,\rho) \cong {}^\sigma \st(k,\rho) \cong \st(k,{}^\sigma\rho)$.
The result now follows from \cite[Theorem 9.7(b)]{zelevinsky1980}, which gives us that $\st(k,\rho) \cong \st(k,{}^\sigma\rho)$ if and only if $\rho \cong {}^\sigma \rho$.
\end{proof}

\begin{prop}\label{cor-st-not-Gal-inv}
For $n\geq 2$, there exist infinitely many unitary twists of the  Steinberg representation $\st_n$ of $\GL_n(E)$ that are not Galois invariant.
\end{prop}

\begin{proof}
Let $\chi:E^\times \rightarrow \C^\times$ be a (unitary) character of $E^\times$.
By \cite[Theorem 9.7(b)]{zelevinsky1980}, $\chi\st_n \cong {}^\sigma (\chi \st_n)$ if and only if $\chi = {}^\sigma \chi$.
We have that ${}^\sigma \chi = \chi$ if and only if $\chi$ is trivial on the kernel of the norm map $N_{E/F}: E^\times \rightarrow F^\times$.
Note that $\ker N_{E/F}$ is a non-trivial closed subgroup of $E^\times$.
We can extend any non-trivial unitary character of $\ker N_{E/F}$ to $E^\times$ to obtain a unitary character $\chi$ of $E^\times$ such that ${}^\sigma \chi \neq \chi$.
Given a unitary character of $\ker N_{E/F}$ there are infinitely many distinct extensions to $E^\times$.
\end{proof}

The following is the main ingredient needed to prove \Cref{prop-exist-reg-dist-ds-unitary}.

\begin{thm}[Hakim--Murnaghan]\label{thm-sc-unitary}
For $n\geq 1$, there exist infinitely many distinct equivalence classes of Galois invariant (respectively, non-Galois invariant) irreducible supercuspidal representations of $\GL_n(E)$.
\end{thm}

\begin{proof}
If $n=1$, argue as in the proof of \Cref{cor-st-not-Gal-inv}.
For $n\geq 2$, use \cite[Proposition 10.1]{murnaghan2011a} to obtain the existence of infinitely many pairwise inequivalent Galois invariant irreducible supercuspidal representations of $\GL_n(E)$.
To complete the proof, apply \cite[Theorem 1.1]{Hakim--Murnaghan2002}.
\end{proof}

\begin{proof}[Proof of \Cref{prop-exist-reg-dist-ds-unitary}]
 If $n$ is prime, then by \Cref{cor-st-not-Gal-inv} there exist infinitely many twists of the Steinberg representation $\st_n$ of $\GL_n(E)$ that are not Galois invariant. 
If $n$ is composite, then \Cref{cor-st-not-Gal-inv} still applies; however, by \Cref{prop-gen-st-unitary} and \Cref{thm-sc-unitary} there are infinitely many classes of non-Galois invariant generalized Steinberg representations of $\GL_n(E)$.
\end{proof}

\begin{rmk}
A representation $(\pi,V)$ is $H$-relatively supercuspidal if and only if the $\lambda$-relative matrix coefficients of $\pi$ are compactly supported modulo $Z_GH$, for every nonzero $\lambda \in \Hom_H(\pi,1)$.
If further assume that $\tau'$ is supercuspidal in \Cref{thm-unitary-rds}, then $\pi = \iota_Q^G\tau$ is a non-supercuspidal $H$-relatively supercuspidal representation of $G$ (\textit{cf.}~\cite[Corollary 6.7]{smith2018}).  A proof of this result can be obtained by a slight modification to proof of \Cref{thm-unitary-rds}.
Indeed, when $P\cap{}^yL$ is proper in ${}^yL$, the subquotients $\mathcal F^y_N(\tau)$ of the Jacquet module vanish since $\tau$ is supercuspidal.  
Otherwise, $\mathcal F^y_N(\tau)$ cannot be $M^\theta$-distinguished.
By  \Cref{non-dist-gen-eig-sp}, $\lambda_N = 0$, for every proper $\theta$-split parabolic subgroup $P$ of $G$ and any $\lambda \in \Hom_H(\pi,1)$.
Finally, by a result of Kato and Takano \cite[Theorem 6.2]{kato--takano2008}, $\pi$ is relatively supercuspidal. Moreover, since $\pi$ is parabolically induced, $\pi$ is not supercuspidal.
This modification of \Cref{thm-unitary-rds} can be obtained by more direct methods; see, for instance, the work of Murnaghan \cite{murnaghan2017a}.
\end{rmk}

\subsection{Exhaustion of the discrete spectrum}\label{sec-exhaustion}
In his 2017 Cours Peccot, Rapha\"el Beuzart--Plessis announced Plancherel formulas for the two $p$-adic symmetric spaces $\GL_n(F) \backslash \GL_n(E)$ and ${\bf U}_{n,E/F}(F) \backslash \GL_n(E)$, where ${\bf U}_{n,E/F}(F)$ is a quasi-split unitary group \cite{beuzart-plessis--CP}.  The two Plancherel formulas are realized in terms of the appropriate base change maps.
Both results are obtained by a comparison of local relative trace formulas.
Presently, we are concerned with the second case and only when $n$ is even.

As above, let $G=\GL_{2n}(E)$ and $H = {\bf U}_{E/F,w_\ell}(F)$, where $n\geq 2$.  Building on the work of Jacquet \cite{jacquet2001} and Feigon--Lapid--Offen \cite{feigon--lapid--offen2012}, Beuzart-Plessis has shown that the Plancherel formula for $H\backslash G$ is the push-forward of the Whittaker--Plancherel formula for $\GL_{2n}(F)$ via quadratic base change.
As a consequence, the discrete spectrum of $H\backslash G$ consists of the quadratic base changes of the discrete series of $\GL_{2n}(F)$.

Let $\irr(\GL_n(E))$ denote the set of equivalence classes of irreducible admissible representations of $\GL_n(E)$ (likewise for $\GL_n(F)$), and let $\irr^\sigma(\GL_n(E))$ denote the subset $ \{ \pi \in \irr(\GL_n(E)) : \pi \cong {}^\sigma \pi \}$  of Galois invariant representations.
Quadratic base change $\operatorname{bc}:\irr(\GL_n(F)) \rightarrow \irr^\sigma(\GL_n(E))$ maps (classes of) irreducible representations of $\GL_n(F)$ to (classes of) irreducible Galois invariant representations of $\GL_n(E)$.
Moreover, the map $\operatorname{bc}:\irr(\GL_n(F)) \rightarrow \irr^\sigma(\GL_n(E))$ is surjective.
Let $\eta_{E/F}:F^\times \rightarrow \C^\times$ be the quadratic character associated to the extension $E/F$ by local class field theory.
For any $\pi' \in \irr(\GL_n(F))$, we have that $\operatorname{bc}(\pi') = \operatorname{bc}(\pi'\otimes\eta_{E/F})$.
We refer the reader to \cite[Chapter 1, Section 6]{Arthur--Clozel-book} for more information about quadratic base change and its properties; in particular, \cite[Chapter 1, Theorem 6.2]{Arthur--Clozel-book} summarizes the basic properties of base change for tempered representations.

In the language of \cite{feigon--lapid--offen2012}, the RDS representations $\pi = \iota_Q^G(\tau' \otimes {}^\sigma \tau')$, with $\tau'\ncong{}^\sigma\tau'$, constructed in \Cref{thm-unitary-rds} are \textit{totally $\sigma$-isotropic}, that is, the cuspidal support of $\pi$ is a tensor product of non-Galois invariant supercuspidal representations (see \Cref{prop-gen-st-unitary}).
Moreover, this means that $\pi$ is the base change of a unique discrete series representation $\pi' \cong \pi'\otimes \eta_{E/F}$ of $\GL_{2n}(F)$, and $\pi$ is not distinguished by the non-quasi-split unitary group \cite[Theorem 0.2, Lemma 3.4]{feigon--lapid--offen2012}.

The following theorem frames the work of Beuzart-Plessis in terms of \Cref{thm-unitary-rds} and gives a complete description of $L^2_{\operatorname{disc}}(H\backslash G)$.

\begin{thm}\label{thm-exhaustion}
Let $\pi$ be a relative discrete series representation for the quotient ${\bf U}_{E/F}(F)\backslash\GL_{2n}(E)$.  Then $\pi$ is either an ${\bf U}_{E/F}(F)$-distinguished discrete series representation of $\GL_{2n}(E)$, or $\pi$ is equivalent to a representation of the form constructed in \Cref{thm-unitary-rds}.
\end{thm}

\begin{proof}
Beuzart-Plessis has shown that the relative discrete series representations for $H\backslash G$ are precisely the images of the discrete series of $\GL_{2n}(F)$ under quadratic base change \cite{beuzart-plessis--CP}.
Let $\pi' \in \irr(\GL_{2n}(F))$ be a discrete series representation of $\GL_{2n}(F)$.
By \cite[Proposition 6.6]{Arthur--Clozel-book}, $\pi = \operatorname{bc}(\pi') \in \irr^\sigma(G)$ is a discrete series representation of $G$ if and only if $\pi' \ncong \pi'\otimes \eta_{E/F}$.
In this case, $\pi = \operatorname{bc}(\pi')$ is an $H$-distinguished discrete series representation of $G$ \cite[Corollary 13.5]{feigon--lapid--offen2012} and $\pi$ is known to be relatively discrete \cite[Proposition 4.10]{kato--takano2010}.
Otherwise, it must be the case that $\pi' \cong \pi'\otimes \eta_{E/F}$.
If $\pi' \in \irr(\GL_{2n}(F))$ is a discrete series representation such that $\pi' \cong \pi'\otimes \eta_{E/F}$, 
then there exists a (non-unique) discrete series representation $\tau' \in \irr(\GL_n(E))$ such that $\tau' \ncong {}^\sigma \tau'$ and $\operatorname{bc}(\pi') = \iota_{P_{(n,n)}}^G(\tau' \otimes {}^\sigma \tau')$ is equivalent to a relative discrete series representation constructed in \Cref{thm-unitary-rds} (\textit{cf.}~\cite[Section 3.2, Lemma 3.4]{feigon--lapid--offen2012}).\footnote{In this case, the representation $\pi'$ is the \textit{automorphic induction} of the discrete series $\tau'$, see \cite[Section 3.2]{feigon--lapid--offen2012}.}
\end{proof}

\section{A technical lemma}\label{sec-technical}
Finally, we give two technical results required to prove \Cref{unitary-case2-exp-rel-cass}, which allows us to reduce the Relative Casselman's Criterion for $\pi = \iota_Q^G\tau$ to the usual Casselman's Criterion for $\tau$.  
The set $A_{M \cap {}^y L}^{- {}^yL} \setminus A_{M\cap {}^y L}^1A_{{}^y L}$ that appears in \Cref{technical-torus-nbhd-Unitary-general} is the dominant part of $A_{M\cap{}^yL}$ in $M\cap{}^yL$, and is precisely the cone on which we must consider the exponents of $\tau$ in order to apply Casselman's Criterion. 
\Cref{technical-torus-nbhd-Unitary} is the analogue of \cite[Lemma 8.4]{smith2018} and the proof is essentially the same (\textit{cf.}~\cite[Lemma 5.2.13]{smith-phd2017}).  
In the present setting, we must also consider non-$\Delta_0$-standard $\theta$-split parabolic subgroups in our analysis of the exponents of $\pi$.
In \Cref{technical-torus-nbhd-Unitary-general}, we will explain how to adapt \Cref{technical-torus-nbhd-Unitary} to handle the non-standard case.

In order to discuss Casselman's Criterion for the inducing data of $\pi=\iota_Q^G\tau$ we use the following notation.
If $\Theta_1 \subset \Theta_2 \subset \Delta_0$, then we define
\begin{align*}
A_{\Theta_1}^- = \{ a \in A_{\Theta_1} : |\alpha(a)|\leq 1,\ \text{for all} \ \alpha \in \Delta_0\setminus {\Theta_1}\}
\end{align*}
and
\begin{align*}
A_{\Theta_1}^{-\Theta_2} = \{ a \in A_{\Theta_1} :  |\beta(a)|\leq 1,\ \text{for all} \ \beta \in \Theta_2 \setminus {\Theta_1}\}.
\end{align*}
The set $A_{\Theta_1}^-$ is the dominant part of $A_{\Theta_1}$ in $G$, while $A_{\Theta_1}^{-\Theta_2}$ is the dominant part of $A_{\Theta_1}$ in $M_{\Theta_2}$.
%
%
\begin{lem}\label{technical-torus-nbhd-Unitary}
Let $P_\Theta$, given by $\Theta \subset \Delta_0$, be any maximal $\theta$-split $\Delta_0$-standard parabolic subgroup.  
Let $w\in[W_\Theta \backslash W_0 / W_\Omega]$ such that $M_{\Theta}\cap {}^w M_\Omega = M_{\Theta \cap w\Omega}$ is a proper Levi subgroup of ${}^w M_\Omega = M_{w\Omega}$. 
We have the containment
\begin{align}\label{eq-unitary-std-cone}
S_\Theta^- \setminus S_\Theta^1 S_{\Delta_0} \subset A_{\Theta \cap w\Omega}^{- w\Omega} \setminus A_{\Theta \cap w\Omega}^1A_{w\Omega}.
\end{align}
\end{lem} 
Recall that $S_\Theta^1 = S_\Theta(\of)$ and $A_{\Theta\cap w\Omega}^1 = A_{\Theta\cap w\Omega}(\of)$.
Let $P = MN$ be a maximal $\theta$-split parabolic subgroup of $G$.  By \Cref{KT08-lem-2.5}, there exists $g\in (\Hbf \mathbf T_0)(F)$ such that $P = g P_\Theta g^{-1}$, where $P_\Theta$ is a $\Delta_0$-standard maximal $\theta$-split parabolic subgroup. 
We may take $S_M = g S_\Theta g^{-1}$ and then we have $S_M^- = g S_\Theta^- g^{-1}$.  
Let $y\in P \backslash G / Q$, given by $y =g w \gamma$, where $w\in [W_{\Theta} \backslash W_0 / W_\Omega]$.
We observe that ${}^yL = g(M_{w\Omega}) g^{-1}$.  In particular, $M \cap {}^y L = g (M_{\Theta \cap w\Omega}) g^{-1}$ and 
$A_{M \cap {}^y L} = g (A_{\Theta \cap w\Omega})g^{-1}$.  The dominant part of the torus $A_{M \cap {}^y L}$ will be denoted by $A_{M \cap {}^y L}^{- {}^yL}$ and is determined by the simple roots $gw\Omega$ of the maximal $(\theta,F)$-split torus ${}^gA_0$ in ${}^yL$.
\begin{lem}\label{technical-torus-nbhd-Unitary-general}
Let $P= MN$ be any maximal $\theta$-split parabolic subgroup of $G$ with $\theta$-stable Levi $M$ and unipotent radical $N$. 
Choose a maximal subset $\Theta$ of $\Delta_0$ and an element $g\in (\Hbf \mathbf T_0)(F)$ such that $P = gP_\Theta g^{-1}$.
 Let $y=g w \gamma\in P \backslash G / Q$, where $w\in [W_{\Theta} \backslash W_0 / W_\Omega]$, such that $M\cap {}^y L$ is a proper Levi subgroup of ${}^y L$. Then we have the containment 
\begin{align}\label{eq-unitary-gen-cone}
S_M^- \setminus S_M^1 S_{G} \subset A_{M \cap {}^y L}^{- {}^yL} \setminus A_{M\cap {}^y L}^1A_{{}^y L}.
\end{align}
\end{lem}

\begin{proof}
The $(\theta,F)$-split component $S_\Theta$ of $M_\Theta$ is equal to its $F$-split component $A_\Theta$; moreover, we have that $S_M = A_M$.
We also have that $S_M^- = g S_\Theta^- g^{-1}$ and $S_M^1 = g S_\Theta^1 g^{-1}$; moreover,  since $S_G = S_{\Delta_0}$ is central in $G$ we obtain 
\begin{align}\label{conj-std-split-cone}
S_M^- \setminus S_G S_M^1 = g (S_\Theta^-) g^{-1} \setminus S_G g (S_\Theta^1) g^{-1}.
\end{align}
By \Cref{technical-torus-nbhd-Unitary}, we have that
\begin{align*}
S_\Theta^- \setminus S_\Theta^1 S_{\Delta_0} \subset A_{\Theta \cap w\Omega}^{-w\Omega} \setminus A_{\Theta \cap w\Omega}^1A_{w\Omega}.
\end{align*}
By the equality in \eqref{conj-std-split-cone}, it suffices to show that
\begin{align}\label{conj-std-split-cone-Levi}
A_{M \cap {}^y L}^{- {}^yL} \setminus A_{M\cap {}^y L}^1A_{{}^y L}
&= g\left(A_{\Theta \cap w\Omega}^{-w\Omega}\right)g^{-1} \setminus g\left(A_{\Theta \cap w\Omega}^1\right)g^{-1} \ g\left(A_{w\Omega}\right)g^{-1}.
\end{align}
Indeed, if \eqref{conj-std-split-cone-Levi} holds, then we have
\begin{align*}
S_M^- \setminus S_G S_M^1 &= g S_\Theta^- g^{-1} \setminus S_G g S_\Theta^1 g^{-1} \\
& \subset g\left(A_{\Theta \cap w\Omega}^{-w\Omega}\right)g^{-1} \setminus g\left(A_{\Theta \cap w\Omega}^1\right)g^{-1} \ g\left(A_{w\Omega}\right)g^{-1} 
	& \text{(\Cref{technical-torus-nbhd-Unitary})}\\
&  = A_{M \cap {}^y L}^{- {}^yL} \setminus A_{M\cap {}^y L}^1A_{{}^y L},
\end{align*}
as claimed.
The truth of \eqref{conj-std-split-cone-Levi} immediately follows from how we determine the dominant part of $A_{M\cap {}^yL}$.
As above, we have that $M \cap {}^y L = g (M_{\Theta \cap w\Omega}) g^{-1}$ and 
$A_{M \cap {}^y L} = g (A_{\Theta \cap w\Omega})g^{-1}$.
Moreover, $A_{M \cap {}^y L}^1 = g (A_{\Theta \cap w\Omega}^1)g^{-1}$.
Given a root $\alpha \in \Phi_0$ we obtain a root $g\alpha$ of ${}^gA_0$ in $G$ by setting $g\alpha = \alpha \circ \Int g^{-1}$, as usual.
Explicitly, we have that
\begin{align}\label{dominant-in-Levi-cone}
A_{M\cap {}^yL}^{-{}^yL} = \{ a \in A_{M\cap {}^yL} : |g\beta(a)|\leq 1, \ \beta \in w \Omega \setminus \Theta \cap w\Omega \}.
\end{align} 
In fact, we have that $M\cap {}^yL$ is determined (as Levi subgroup of ${}^y L$) by the simple roots 
 $g(\Theta \cap w\Omega) \subset gw\Omega$ 
of ${}^gA_0$ in ${}^yL = gwM_\Omega$. 
It is immediate that 
\begin{align}\label{dominant-in-Levi-cone-containment}
A_{M \cap {}^y L}^{-{}^y L}  = g \left(A_{\Theta \cap w\Omega}^{-w\Omega}\right) g^{-1},
\end{align}
from which \eqref{conj-std-split-cone-Levi} follows, completing the proof of the lemma.
\end{proof}
\addcontentsline{toc}{section}{Acknowledgements}
\subsection*{Acknowledgements}
I would like to thank my doctoral advisor Fiona Murnaghan for her support and guidance throughout the initial work on this project.
Thank you to Rapha\"el Beuzart-Plessis for providing copies of his 2017 Cours Peccot lecture notes and for his generous explanations.
I would also like to thank the referee for their valuable comments and, in particular, for the suggested improvements to \Cref{prop-exist-reg-dist-ds-unitary,cor-st-not-Gal-inv}. 
Finally, thank you to the editor for directing my attention to the forthcoming work of Beuzart-Plessis and suggesting the inclusion of \Cref{sec-exhaustion}.  

\bibliographystyle{amsplain}
\bibliography{jerrod-refs}
%

\end{document}